%% file: bcgm.tex
\documentclass[a4paper,12pt]{article}

\usepackage{mathtools}
\usepackage[all]{xy}
\usepackage{float}
\usepackage{amsmath,amssymb,amsfonts,amsthm}
\usepackage{enumerate}
\usepackage{epstopdf}
\usepackage{color,fancybox}
\usepackage{url}
\usepackage[backref]{hyperref}
\usepackage{graphicx,psfrag,epsfig}
\usepackage{bm}
\usepackage{tikz}
\usetikzlibrary{positioning,arrows,calc,decorations.text,scopes,intersections,decorations.markings}
\include{macros}

\DeclareMathOperator*{\argmin}{argmin}

\numberwithin{equation}{section}

\begin{document}

\begin{center}
  {\sc \Large Recursive Estimation of a Failure Probability for a Lipschitz Function}
  \vspace{1cm}
\end{center}

{\bf Lucie Bernard}\\
{\it IDP, Universit\'e de Tours, France }\\
\textsf{lucie.bernard@live.fr}
\bigskip

{\bf Albert Cohen}\\
{\it LJLL, Sorbonne Universit\'e, France }\\
\textsf{cohen@ann.jussieu.fr}
\bigskip

{\bf Arnaud Guyader\footnote{Corresponding author.}}\\
{\it LPSM, Sorbonne Universit\'e \& CERMICS, France }\\
\textsf{arnaud.guyader@upmc.fr}
\bigskip

{\bf Florent Malrieu}\\
{\it IDP, Universit\'e de Tours, France }\\
\textsf{florent.malrieu@univ-tours.fr}
\bigskip

\medskip

\begin{abstract} Let $g:\Omega=[0,1]^d\to\R$ denote a Lipschitz function that can be evaluated at each point, but at the price of a heavy computational time. Let $X$ stand for a random variable with values in $\Omega$ such that one is able to simulate, at least approximately, according to the restriction of the law of $X$ to any subset of $\Omega$. For example, thanks to Markov chain Monte Carlo techniques, this is always possible when $X$ admits a density that is known up to a normalizing constant. In this context, given a deterministic threshold $T$ such that the failure probability $p:=\P(g(X)>T)$ may be very low, our goal is to estimate the latter with a minimal number of calls to $g$. In this aim, building on Cohen \textit{et al.}\ \cite{MR3197124}, we propose a recursive and optimal algorithm that selects on the fly areas of interest and estimate their respective probabilities. \medskip

   \medskip


\end{abstract}

\tableofcontents


\section{Introduction}\label{intro}

Let $g:\Omega=[0,1]^d\to\R$ denote a function that can be evaluated
at any point $x\in\Omega$. Then, considering a random variable $X$ with values in $\Omega$ that we can easily simulate, we want to estimate the so-called failure probability 
$$p:=\P(g(X)>T),$$ 
where $T$ is a fixed threshold such that $p$ is strictly positive but possibly very low. 
We are motivated by applications where each evaluation of the 
function $g$ at a given $x\in \Omega$ is costly. For example, it could
be the result of a numerical simulation or of a physical experiment, that has to be
repeated for each new value of $x$. Therefore, one would like to limitate as much
as possible the number
of queries $x\mapsto g(x)$.\medskip

In this framework, a naive Monte Carlo method consists in simulating $n$ independent and identically distributed (i.i.d.) random variables $X_1,\dots,X_n$ with the same law as $X$, and considering the estimator
$$p_n:=\frac{1}{n}\sum_{i=1}^n\un_{g(X_i)>T}.$$
Since the random variables $\un_{g(X_i)>T}$ are i.i.d.\ with Bernoulli law $\mathcal{B}(p)$, this estimator is unbiased, strongly consistent, and satisfies the following central limit theorem:
$$\sqrt{n}(p_n-p)\xrightarrow[n\to\infty]{d}\mathcal{N}(0,p(1-p)).$$
However, this is an asymptotic result that is of no practical interest unless $n$ is of order $1/p$.  Indeed, if $n\ll 1/p$, as is the case in the situations we have in mind, then most of the time $p_n=0$ and this estimator is useless.\medskip

To circumvent this issue, the purpose of variance reduction techniques is to make the rare event less rare and, in turn, decrease the previous asymptotic variance, that is $\sigma^2=p(1-p)$. For example, instead of simulating  according to the law $\mu$ of $X$, the idea of Importance Sampling is to consider an auxiliary distribution $\tilde\mu$ such that, if $\tilde X\sim\tilde\mu$, the event $\{g(\tilde X)>T\}$ is not rare. If this is possible, one then just has to simulate $\tilde X_1,\dots,\tilde X_n$ i.i.d.\ according to $\tilde\mu$, and consider the estimator
$$\tilde p_n:=\frac{1}{n}\sum_{i=1}^n\frac{d\mu}{d\tilde\mu}(\tilde X_i)\un_{g(\tilde X_i)>T},$$
where $\frac{d\mu}{d\tilde\mu}$ stands for the Radon-Nikodym derivative of $\mu$ w.r.t.\ $\tilde\mu$. This technique has been widely applied in practice and may indeed lead to dramatic variance reductions. However, it requires a lot of information about both the failure domain 
\begin{equation}
\label{failure}
F:=\{x\in\Omega\; : \; g(x)>T\},
\end{equation}
and the law $\mu$ in order to find a relevant instrumental distribution $\tilde\mu$. There is a huge amount of literature on this topic. Among the first references, we can mention the paper by Kahn and Harris in particle physics \cite{kahn}, while the application to structural safety dates back at least to Harbitz \cite{harbitz}. We refer for example to the monograph \cite{bucklew04} for details.\medskip

Another classical variance reduction technique is Importance Splitting, introduced by Kahn and Harris \cite{kahn}. The principle is to consider several intermediate levels $-\infty=L_0<L_1<\dots<L_K=L$ such that each conditional probability $p^{(k)}:=\P(g(X)>L_k|g(X)>L_{k-1})$ is not small, and to apply the corresponding Bayes formula $p=p^{(1)}\dots p^{(K)}$. Accordingly, if $\hat p^{(k)}_n$ is an estimator of $p^{(k)}$, then a natural estimator for $p$ is simply 
$$\hat p_n=\hat p^{(1)}_n\dots \hat p^{(K)}_n.$$
In our specific context, this is the purpose of Subset Simulation \cite{Au2001263,au:901} and Adaptive Multilevel Splitting \cite{cdfg, cg4,cgr}. This is particularly suitable when $X$ has a density $f_X$ that is known up to a normalizing constant, like for example in Bayesian statistics and statistical physics, for one may then apply Markov Chain Monte Carlo (MCMC) techniques to estimate each intermediate probability $p_k$. As explained in \cite{ghm}, the best asymptotic variance that one can expect through splitting techniques is $s^2=p^2\log (p^{-1})$, which is indeed much lower than $\sigma^2=p(1-p)$. Nonetheless, if $t$ stands for the number of steps of each Markov chain constructed at each step $k$, this necessitates about $tn\log(n)\log(p^{-1})$ calls to $g$, which is much larger than the number $n$ of calls required for a naive Monte Carlo estimator. Therefore, when the simulation budget is severely limited, we can not directly apply these splitting techniques, even if we will recycle some of their ingredients in what follows.\medskip

In uncertainty quantification, a standard approach is to make more or less agressive assumptions on the failure domain $F$ and/or the function $g$.  One may trace back this idea to First (respectively Second) Order Reliability Methods, or FORM (respectively SORM) for short. In a nutshell, they assume that one can rewrite the probability of interest as $p=\P(L(Z)<0)$, where $Z$ stands for a standard Gaussian random vector in dimension $d$. Denoting $z^\star:=\argmin\{\|z\|_2, L(z)=0\}$ the so-called most probable point, the idea is to approximate $p$ by the probability that $Z$ falls in the neighborhood of $z^\star$. We refer to \cite{Livre_opti} and references therein for more details.\medskip

Alternatively, a widespread Bayesian framework consists in assuming that the function $g$ is the realization of a Gaussian random field, defined as a prior model. Conditionally on observed values of the function, the posterior model is still Gaussian. Its mean function provides a surrogate model used to approximate $g$ while the variance represents the uncertainty of the model (see, e.g., \cite{Rasmussen06}). It is then possible to construct sequential sampling strategies to estimate the probability of failure. It basically consists in determining each new evaluation of $g$ by minimizing a criterion that ensures that the precision of the considered estimator is improved. For instance, one may apply Stepwise Uncertainty Reduction strategies, which are formalized in \cite{MR2909621} in this Bayesian framework. 
Combined with Subset Simulation, this approach can also be found in \cite{Bect_2017} for the estimation of very small probabilities. Note that this Gaussian process modelling approach corresponds to an assumption on the regularity of $g$, notably through the choice of the correlation function (see, e.g., \cite{Rasmussen06}).\medskip

 Let us also finally mention that polynomial chaos expansions represent another set of popular non-intrusive metamodelling techniques. The principle is to approximate the mapping $g$ by a series of multivariate polynomials which are orthogonal with respect to the distributions of the input random variables $X_1,\dots,X_d$ (see, e.g., \cite{MR3350081} and references therein). In particular, it allows one to compute analytically Sobol' indices, which are a standard tool in uncertainty quantification.\medskip

Here we do not adopt a Bayesian/metamodelling approach. Concerning the function $g$, we suppose that it is $L$-Lipschitz, with $L$ known, and satisfies a so-called level set condition (see Assumption \ref{hyp:level-set}). As for the law of $X$, we assume that it admits a bounded density $f_X$ that is known up to a normalizing constant, or that we are able to simulate at least approximately according to the restriction of $f_X$ to any subset of $\Omega$. In this framework, building on \cite{MR3197124}, we show that the failure probability $p$ admits a lower (resp.\ upper) bound $p^{-}_n$ (resp.\ $p^{+}_n$ ) based on $n$ calls to $g$, and such that the approximation error satisfies, for $d\geq 2$,
\begin{equation}\label{lkscnlk}
E_n:=p^{+}_n-p^{-}_n\leq C n^{-\frac{1}{d-1}}.
\end{equation}
Even if this rate of convergence is classic in deterministic numerical integration, one may notice that the quantity of interest
$$p=\int_\Omega\mathbf{1}_{g(x)>T}f_X(x)dx$$
is the integral of a non regular function, which makes the problem non trivial.
In fact, we prove in Section \ref{dcljanlcna} that this rate is optimal, meaning that under this set of assumptions, no algorithm based on $n$ calls to $g$ can achieve a better approximation error.\medskip
 
Nevertheless, besides $n$ calls to $g$, our algorithm requires the sequential evaluation of probabilities of the form $\P(X\in Q)$, where $Q$ stands for a generic dyadic subcube of $\Omega$. It is generally impossible to do this exactly, but in many situations of interest we may apply standard MCMC techniques to estimate these probabilities with an arbitrary small (random) error. More explicitly, we propose to adopt here the same idea as in the abovementioned splitting techniques, by generating for each $Q$ a sample of size $N$ that is approximately i.i.d.\ according to the restriction of the law of $X$ to $Q$.\medskip

Putting all pieces together, we propose a sequential algorithm with global stochastic error
\begin{equation}\label{aoecoaezc}
|\hat E_n^N|\leq C n^{-\frac{1}{d-1}}+O_p(1/\sqrt{N}).
\end{equation}
We point out that, in the latter, since the second term does not require any supplementary evaluation of $g$, it can easily be made arbitrarily small, so that only the first one matters and, as already explained, this first term is optimal for our set of assumptions.\medskip

The article is organized as follows. Section \ref{ajncpazc} gives in more details the assumptions and the main results of this work. Section \ref{apecjapijdapid} explains the deterministic algorithm that allows us to reach the approximation error $E_n$ in \eqref{lkscnlk}, while the proof of its optimality is deferred to Section \ref{dcljanlcna}. Section \ref{khckasjbc} makes more explicit the term $O_p(1/\sqrt{N})$ in \eqref{aoecoaezc} and provides asymptotic confidence intervals for our estimators. All of these results are illustrated on a toy example in Section \ref{ksjbksjcxksj}, and the proof of Theorem \ref{th:esti-p} is detailed in Section \ref{qlsjcnlnc}.

\section{Assumptions and main results}
\label{ajncpazc}

Let $X$ be a random variable on $\Omega=[0,1]^d$ with $d\in\dN^\star$ and $g: \Omega\to \dR$. 
For a given threshold $T\in\dR$, let us denote by $F$ the failure domain and $p$ the 
failure probability, i.e., 
\[
F=\BRA{x\in \Omega\,:\,g(x)>T}
\quad
\text{and} 
\quad 
p=\dP(X\in F)=\dP(g(X)>T).  
\]
We intend to present and analyse an algorithm to estimate this failure probability as precisely as possible for a given total number $n$ of calls to~$g$. In all what follows, the upcoming assumptions will be of constant use. 

\begin{Ass}[Absolute continuity of the distribution of $X$]\label{kjncllxnlankx}
 The distribution of $X$ on $\Omega$ admits a bounded density function $f_X$ with respect to the Lebesgue measure~$\lambda$. In other words
$$
\|f_X\|_{L^\infty}=K<\infty.
$$
\end{Ass}

\begin{Ass}[Lipschitz smoothness]\label{kdjcnkajsc}
 The function $g$ is assumed to be $L$-Lipschitz with respect to the supremum norm on $\dR^d$, 
 i.e., 
 \[
\ABS{g(x)-g(\tilde x)}\leq L \NRM{x-\tilde x}_\infty, 
\quad  x,\tilde x\in \Omega.  
\]
Equivalently, $\nabla g\in L^\infty(\Omega)$ with
$\|\nabla g(x)\|_1 \leq L$ almost everywhere in $\Omega$.
\end{Ass}

Here, we denote by $\|z\|_p$ the $\ell^p$ norm of a
vector $z\in \R^d$. For the Euclidean norm, we sometime simply write $|z|:=\|z\|_2$.

\begin{Ass}[Level set condition]\label{hyp:level-set}
There exists a constant $M>0$ such that 
 \[
\lambda\PAR{\BRA{x\in \Omega \,:\,\ABS{g(x)-T}\leq \delta}}\leq M\delta,
\quad \delta>0.
\]
\end{Ass}

The constants $L$ and $M$ in Assumptions \ref{kdjcnkajsc} and \ref{hyp:level-set} are jointly coupled. Indeed, since the failure probability $p$ is such that $0<p<1$, there exists $x_T$ such that $g(x_T)=T$, and for all $x\in\Omega$, we have
\[
|g(x)-T|\leq L\|x-x_T\|_\infty\leq L,
\]
so that, if Assumption \ref{hyp:level-set} is satisfied,
\[
1=\lambda\left(\left\{x\in \Omega \,:\,\ABS{g(x)-T}\leq L\right\}\right)\leq ML,
\]
which shows that $ML\geq 1$. We introduce the constant
\begin{equation}
\label{const}
C:=ML,
\end{equation}
which will appear in the error estimates established for the algorithm presented
and analyzed further.

\begin{Rem} The level set Assumption \ref{hyp:level-set}
may be thought as reflecting the fact that the function $g$
is not too much flat in the vicinity of the level set 
$S_T=g^{-1}(\{T\})$. Indeed, when $d=1$, if $x_T$ 
is a point such that $g(x_T)=T$ and
assuming that $g$ is continuously differentiable, then
$g'(x_T)<M^{-1}$ would contradict
Assumption \ref{hyp:level-set} for $\delta$ small enough.
In the case $d\geq 2$, assuming that $g$ is continuously differentiable with $\nabla g(x) \neq 0$ for any $x\in S_T$, then $S_T$ is a compact submanifold of dimension $(d-1)$ and the coarea formula (see, e.g., \cite{evans}, Proposition 3 page 118) says that, for $\delta$ small enough,
\begin{equation}
\lambda\PAR{\BRA{x\in \Omega \,:\,\ABS{g(x)-T}\leq \delta}}=\int_{T-\delta}^{T+\delta}\left(\int_{S_t}\frac{ds}{|\nabla g(s)|}\right)dt,
\label{coarea}
\end{equation}
where  $ds$ stands for the $(d-1)$-dimensional Hausdorff measure on the level set $S_t=g^{-1}(\{t\})$. As a consequence, Assumption \ref{hyp:level-set} is fulfilled 
with constant $M$ for $\delta$ small enough as soon as 
$$
|\nabla g(x)|>\frac {2H} M,
$$
where $H$ is the $(d-1)$-dimensional Hausdorff measure of $S_T$, and therefore for all $\delta$ up to raising the value of $M$.
\end{Rem}

The proof of the following result is housed in Section \ref{aksjxnjzx} for the first part (definition of the algorithm and error rates), and in Section \ref{dcljanlcna} for the second part (optimality).

\begin{The}\label{zdjcnlzdc}
 Under Assumptions \ref{kjncllxnlankx}, \ref{kdjcnkajsc}, and \ref{hyp:level-set}, there exists an algorithm that, based on $n$ calls to $g$, constructs two deterministic bounds $p_n^-\leq p\leq p_n^+$ such that the approximation error $E_n:=p_n^+-p_n^-$ satisfies 
 \begin{itemize}
 \item If $d=1$, $E_n\leq 2CK\ 2^{-\frac{n}{2C}}$.
 \item If $d\geq 2$, $E_n\leq 8C^{\frac{d}{d-1}}K\ n^{-\frac{1}{d-1}}$.
 \end{itemize}
 In addition, these rates of convergence are optimal.
 \end{The}
 
\begin{Rem}
As it will become clear in Section \ref{apecjapijdapid}, the algorithm
that we propose only requires the knowledge of the Lipschitz constant $L$ (or an upper-bound), 
while that of $K$ and $M$ is not needed.
\end{Rem}
 
 The quantities $p_{n}^-$ and $p_{n}^+$ are defined 
 as the measures of certain sets of dyadic cubes that are determined by our algorithm. When $f_X=1$, that is when $X$ is uniformly distributed, this measure can be computed exactly,
 otherwise it may need to be estimated. This requires possibly many samples of $X$, but not any additional call of $g$.
 \medskip
 
 We begin with an idealized situation. The following result is established in Section \ref{lakzlkzx}.
 
\begin{The}\label{amicjpazijcpaic}
If for each dyadic cube $Q$ of $\Omega$, one is able to simulate an $N$ i.i.d.\ sample according to the restriction of the law of $X$ to $Q$, then, without any additional call to $g$,  we can construct two unbiased, strongly consistent and asymptotically Gaussian estimators $p_{n,N}^-$ and $p_{n,N}^+$ of the previous lower and upper bounds, i.e.,
$$\sqrt{N}\left(p_{n,N}^\pm-p_{n}^\pm\right)\xrightarrow[N\to\infty]{d}\mathcal{N}(0,(\sigma_n^\pm)^2),$$
along with consistent estimators $\sigma_{n,N}^-$ and $\sigma_{n,N}^+$ of the latter asymptotic standard deviations.
\end{The}

Unfortunately, it is usually not possible to simulate an $N$ sample that is exactly i.i.d.\ according to the restriction of the law of $X$ to $Q$. However, if the pdf $f_X$ is known up to a normalizing constant (as is the case in many situations of interest), then one can do it at least approximately thanks to a Metropolis-Hastings algorithm. The upcoming proposition gives a flavor of the type of results we obtain in this context.

\begin{Pro}\label{peidpiejcpzejc}
If $f_X$ is continuous strictly positive on $\Omega$, and known up to a normalizing constant, then, without any additional call to $g$, we can construct two estimators $\hat p_{n,N}^-$ and $\hat p_{n,N}^+$ such that, for all $t\in\N^\star$,
$$\P\left(\hat p_{n,N}^\pm=p_{n,N}^\pm\right)\geq \left(1-Ar^t\right)^{m N}.$$
for some constants $A>0$, $0<r<1$, and $m\in\N^\star$. The same result holds true for $\sigma_{n,N}^-$ and $\sigma_{n,N}^+$.
\end{Pro}
 
The proof of this proposition is detailed in Section \ref{zkjcajclacx}.

\section{Approximation error}\label{apecjapijdapid}

\subsection{Neveu's notation}

Let us denote $\calD$ the set of all dyadic subcubes of $\Omega$, and $\calD_j$ the set of all dyadic cubes with 
sidelength $2^{-j}$ for $j\geq 1$. Given a dyadic cube $Q$ in $\calD$, $c_Q$ stands for the 
center of $Q$. Each dyadic cube $Q$ has $2^d$ children numbered from $1$ to $2^d$ 
and each $Q\neq \Omega$ has exactly one parent.\medskip 

In the sequel, we will identify a dyadic cube in $\calD$ 
to a vertex in the infinite $2^d$-regular tree~$\calT$. It will be referred to thanks to Neveu's notation (see \cite{Neveu}): 
the root of the tree, associated to $\Omega$, is denoted by $\emptyset$ and, for any $k\in\dN^\star$ and 
$1\leq u_1,\ldots, u_k\leq 2^d$, the vertex $(u_1,\ldots,u_k)$ is the $u_k^{th}$ child of 
$(u_1,\ldots,u_{k-1})$. A vertex $u=(u_1,\ldots,u_k)$ in $\calT$ is then associated to a 
cube~$Q(u)$ in $\calD$. Notice that the sidelength of $Q(u)$ is $2^{-k}$ where $k$ is 
the depth of $u$ (distance between the root and $u$).\medskip

If $v$ is a vertex in~$\calT$, $\bar{v}$ (resp. $\calC(v)$) 
denotes the parent (resp. the set of the $2^d$ children) of $v$. The vertex $v$ is said to be an 
ancestor of $u$, and we denote $v\leq u$, if $Q(u)\subset Q(v)$ or, equivalently, if $v$ is a 
prefix of $u$. Notice that $u\leq u$. In the sequel, $a(v)$ stands for the set made of 
the ancestors of $v$, including $v$ but excluding the root for convenience. Finally, 
if $v$ and $w$ are two vertices, then $v\wedge w$ stands for the more recent common 
ancestor of $v$ and $w$. \medskip

We say that $\Lambda$ 
is a finite $2^d$-regular tree, if it is a finite subset of $\calT$
such that $u\in \Lambda$ and $v\leq u$ implies that $v\in \Lambda$. For any finite $2^d$-regular tree $\Lambda$ of $\calT$, the leaves (resp. internal vertices) 
of $\Lambda$ are the vertices in $\Lambda$ with no child (resp. with $2^d$ children) in 
$\Lambda$. The depth of $\Lambda$ is defined as the
maximal depth of the vertices in $\Lambda$.
See Figure~\ref{fi:tree} for an illustration. \medskip

For the purpose of our algorithm, the dyadic cubes (or vertices) are labelled according to the following rule that involves the
evaluation of $g$ at their centers. 

 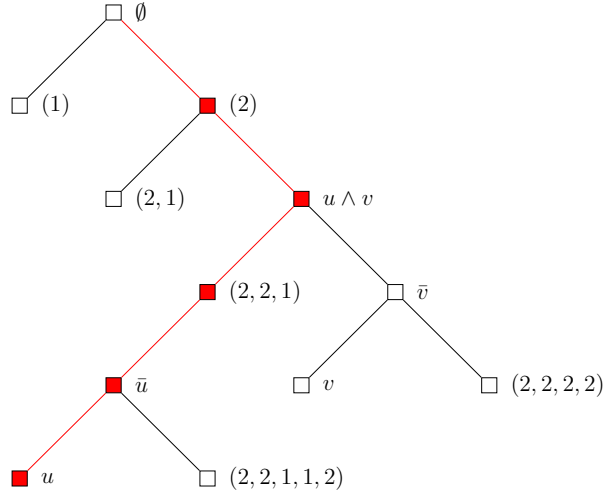
\begin{figure}
 \centering
\resizebox{8cm}{!}{%
\begin{tikzpicture}
\node[draw] (0) at (14,-7) {}; 
\node[right =0.1cm of 0] (0n) {$\emptyset$}; 

\node[below left = 2cm of 0,draw] (1) {}; \draw (0) -- (1);
\node[right =0.1cm of 1] (1n) {$(1)$}; 
\node[below right = 2cm of 0, fill=red, draw] (2) {}; \draw[red] (0) -- (2);
\node[right =0.1cm of 2] (2n) {$(2)$}; 

\node[below left = 2cm of 2,draw] (21) {}; \draw (2) -- (21);
\node[right =0.1cm of 21] (21n) {$(2,1)$}; 
\node[below right = 2cm of 2, fill=red, draw] (22) {}; \draw[red] (2) -- (22);
\node[right =0.1cm of 22] (22n) {$u\wedge v$}; 

\node[below left = 2cm of 22, fill=red, draw] (221) {}; \draw[red] (22) -- (221);
\node[right =0.1cm of 221] (221n) {$(2,2,1)$}; 
\node[below right = 2cm of 22, draw] (222) {}; \draw (22) -- (222);
\node[right =0.1cm of 222] (222n) {$\bar v$}; 

\node[below left = 2cm of 221, fill=red, draw] (2211) {}; \draw[red] (221) -- (2211);
\node[right =0.1cm of 2211] (2211n) {$\bar u$}; 

\node[below left = 2cm of 2211, fill=red, draw] (22111) {}; \draw[red] (2211) -- (22111);
\node[right =0.1cm of 22111] (22111n) {$u$}; 
\node[below right = 2cm of 2211, draw] (22112) {}; \draw (2211) -- (22112);
\node[right =0.1cm of 22112] (22112n) {$(2,2,1,1,2)$}; 

\node[below left = 2cm of 222,draw] (2221) {}; \draw(222) -- (2221);
\node[right =0.1cm of 2221] (2221n) {$v$}; 
\node[below right = 2cm of 222, draw] (2222) {}; \draw (222) -- (2222);
\node[right =0.1cm of 2222] (2222n) {$(2,2,2,2)$}; 
\end{tikzpicture}
}
   \caption{{Example of a $2$-regular finite tree of depth $5$ in dimension 1. The red line represents the set $a(u)$ of 
   the ancestors of the leaf~$u$.}}%
  \label{fi:tree}%
 \end{figure}

\begin{defi}[Label of a cube]\label{def:label}
The dyadic cube $Q$ with side length~$2^{-j}$ and center $c_Q$ is labelled  
\begin{itemize}
\item $\calI$  (inside) if $g(c_Q)> T+L2^{-j-1}$, 
\item $\calO$ (outside) if $g(c_Q)< T-L2^{-j-1}$, 
\item $\calU$ (uncertain) otherwise. 
\end{itemize}
\end{defi}

A cube with label $\calI$ is included in the failure set~$F$. Indeed, for any $Q\in\calD_j$ and any $x\in Q$,  
\[
\ABS{g(x)-g(c_Q)}\leq L\NRM{x-c_Q}_\infty\leq L 2^{-j-1}.
\]
As a consequence, if the label of $Q$ is $\calI$, then, for any $x\in Q$, we have
\[
g(x)\geq g(c_Q)-L2^{-j-1}>T.  
\]
Likewise, a cube with label $\calO$ is included in $F^c:=\Omega\setminus F$. Finally, a cube 
with label $\calU$ may intersect $F$ and/or $F^c$.

\subsection{Recursive construction of relevant trees} 

 
The algorithm starts with $\Lambda(0)=\BRA{\Omega}$, where $\Omega$ has the label $\calU$ 
and the depth of $\Lambda(0)$ is $0$. 
At a given step $k>0$, a finite $2^d$-regular tree 
$\Lambda(k)$ of depth $k$ has been constructed
with the following features:
\begin{itemize}
\item[(i)] internal vertices are all labelled as $\calU$,
\item[(ii)] leaves of depth lower than $k$ are labelled $\calI$ or $\calO$,
\item[(iii)] leaves of depth $k$ can have any label. 
\end{itemize}
Then, the tree $\Lambda(k+1)$ is obtained 
by performing a $2^d$-split on each leaf with label $\calU$
and evaluating $g$ at their center in order to label the new
leaves according to Definition \ref{def:label}. Clearly the new tree
$\Lambda(k+1)$ of depth $(k+1)$ has similar properties (see Figure \ref{fi:ex-tree}). 
\medskip

 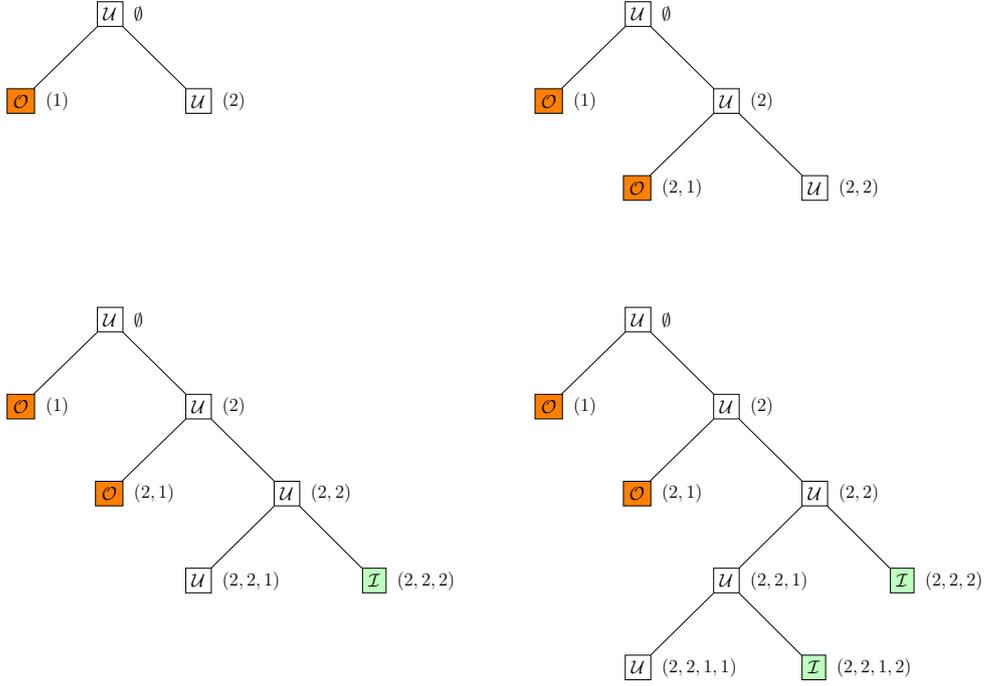
\begin{figure}
\resizebox{13cm}{!}{%
\centering
\begin{tikzpicture}
\node[draw] (0) at (4,0) {$\calU$}; 
\node[right =0.1cm of 0] (0n) {$\emptyset$}; 

\node[below left = 2cm of 0, fill=orange,draw] (1) {$\calO$}; \draw (0) -- (1);
\node[right =0.1cm of 1] (1n) {$(1)$}; 
\node[below right = 2cm of 0, draw] (2) {$\calU$}; \draw (0) -- (2);
\node[right =0.1cm of 2] (2n) {$(2)$}; 

\node[draw] (0) at (16,0) {$\calU$}; 
\node[right =0.1cm of 0] (0n) {$\emptyset$}; 

\node[below left = 2cm of 0, fill=orange,draw] (1) {$\calO$}; \draw (0) -- (1);
\node[right =0.1cm of 1] (1n) {$(1)$}; 
\node[below right = 2cm of 0, draw] (2) {$\calU$}; \draw (0) -- (2);
\node[right =0.1cm of 2] (2n) {$(2)$}; 

\node[below left = 2cm of 2, fill=orange,draw] (21) {$\calO$}; \draw (2) -- (21);
\node[right =0.1cm of 21] (21n) {$(2,1)$}; 
\node[below right = 2cm of 2, draw] (22) {$\calU$}; \draw (2) -- (22);
\node[right =0.1cm of 22] (22n) {$(2,2)$}; 

\node[draw] (0) at (4,-7) {$\calU$}; 
\node[right =0.1cm of 0] (0n) {$\emptyset$}; 

\node[below left = 2cm of 0, fill=orange,draw] (1) {$\calO$}; \draw (0) -- (1);
\node[right =0.1cm of 1] (1n) {$(1)$}; 
\node[below right = 2cm of 0, draw] (2) {$\calU$}; \draw (0) -- (2);
\node[right =0.1cm of 2] (2n) {$(2)$}; 

\node[below left = 2cm of 2, fill=orange,draw] (21) {$\calO$}; \draw (2) -- (21);
\node[right =0.1cm of 21] (21n) {$(2,1)$}; 
\node[below right = 2cm of 2, draw] (22) {$\calU$}; \draw (2) -- (22);
\node[right =0.1cm of 22] (22n) {$(2,2)$}; 

\node[below left = 2cm of 22,draw] (221) {$\calU$}; \draw (22) -- (221);
\node[right =0.1cm of 221] (221n) {$(2,2,1)$}; 
\node[below right = 2cm of 22, fill=green!25, draw] (222) {$\calI$}; \draw (22) -- (222);
\node[right =0.1cm of 222] (222n) {$(2,2,2)$}; 

\node[draw] (0) at (16,-7) {$\calU$}; 
\node[right =0.1cm of 0] (0n) {$\emptyset$}; 

\node[below left = 2cm of 0, fill=orange,draw] (1) {$\calO$}; \draw (0) -- (1);
\node[right =0.1cm of 1] (1n) {$(1)$}; 
\node[below right = 2cm of 0, draw] (2) {$\calU$}; \draw (0) -- (2);
\node[right =0.1cm of 2] (2n) {$(2)$}; 

\node[below left = 2cm of 2, fill=orange,draw] (21) {$\calO$}; \draw (2) -- (21);
\node[right =0.1cm of 21] (21n) {$(2,1)$}; 
\node[below right = 2cm of 2, draw] (22) {$\calU$}; \draw (2) -- (22);
\node[right =0.1cm of 22] (22n) {$(2,2)$}; 

\node[below left = 2cm of 22,draw] (221) {$\calU$}; \draw (22) -- (221);
\node[right =0.1cm of 221] (221n) {$(2,2,1)$}; 
\node[below right = 2cm of 22, fill=green!25, draw] (222) {$\calI$}; \draw (22) -- (222);
\node[right =0.1cm of 222] (222n) {$(2,2,2)$}; 

\node[below left = 2cm of 221,draw] (2211) {$\calU$}; \draw (221) -- (2211);
\node[right =0.1cm of 2211] (2211n) {$(2,2,1,1)$}; 
\node[below right = 2cm of 221, fill=green!25, draw] (2212) {$\calI$}; \draw (221) -- (2212);
\node[right =0.1cm of 2212] (2212n) {$(2,2,1,2)$}; 

%
%
%
%
%
\end{tikzpicture}
}
\caption{Example of a recursive construction of $\Lambda(1),\ldots,\Lambda(4)$ for $d=1$.}%
\label{fi:ex-tree}%
\end{figure}

Denoting $|\Lambda(k)|$ the cardinal of $\Lambda(k)$, the number $n_k$ of evaluations of $g$ that is involved in the construction of $\Lambda(k)$ is therefore given by $n_0=0$ and, for all $k\geq 1$,
$$
n_k=|\Lambda(k)|-1,
$$
since the evaluation at the center of $\Omega$ is useless when $p>0$. The following result gives an upper bound on this number. Recall that $C$ is the constant defined by \eqref{const}.
		
\begin{Pro}\label{prop:nk}
Let $k\geq 0$. If $d=1$, the number $n_k$ of evaluations of $g$ satisfies
\[
n_k\leq 2Ck.
\]  
If $d\geq 2$, then we have 
\[
n_k\leq 4C\ 2^{(d-1)k}.
\]  
\end{Pro}

\begin{proof}
 Since $n_0=0$, the result is clear for $k=0$. Therefore, let us consider the case where $k\geq 1$. For any $0\leq j\leq k-1$, let us denote by $\calU(j)$ the set of leaves 
 of $\Lambda(j)$ with label $\calU$ (see Figure \ref{fi:ex-tree}). Recall from Definition \ref{def:label} that this set is made of dyadic cubes with 
 side length~$2^{-j}$ such that, for any $Q\in\calU(j)$, 
 \[
\ABS{g(c_Q)-T}\leq \frac{L}{2^{j+1}}. 
\]
As a consequence, for any $x\in Q\in\calU(j)$,  Assumption \ref{kdjcnkajsc} gives
\[
\ABS{g(x)-T}\leq \ABS{g(x)-g(c_Q)}+\ABS{g(c_Q)-T}\leq \frac{L}{2^{j}}. 
\]
This ensures that 
\[
\bigcup_{Q\in \calU(j)} Q \subset \BRA{x\in\Omega\ :\ \ABS{g(x)-T}\leq \frac{L}{2^{j}}}. 
\]
Since the volume of each cube in $\calU(j)$ is $2^{-jd}$, this yields
\[
\frac{\ABS{\calU(j)}}{2^{jd}}\leq \lambda\PAR{ \BRA{x\in\Omega\ :\ \ABS{g(x)-T}\leq \frac{L}{2^{j}}}},
\]
with the understanding that $\ABS{\calU(j)}$ is the cardinal of $\calU(j)$. Thanks to Assumption~\ref{hyp:level-set}, we get that 
\begin{equation}\label{eq:card-Uj}
 \ABS{\calU(j)} \leq \PAR{\frac{C}{L} \frac{L}{2^{j}}} 2^{jd}=C 2^{j(d-1)}=:\mu_j.
\end{equation}
Thus, the construction of $\Lambda(j+1)$ requires at most $2^d\mu_j$ evaluations of $g$. As a 
consequence, we can bound the total number $n_k$ of calls to $g$ to construct $\Lambda(k)$ as follows: 
\[
n_k\leq \sum_{j=0}^{k-1}2^d\mu_j. 
\]
If $d\geq 2$, we are led to
\[
n_k\leq C2^d\frac{2^{(d-1)k}-1}{2^{d-1}-1}\leq 4C\ 2^{(d-1)k}.
\]
In the case $d=1$, we obtain $n_k\leq  2Ck$.
\end{proof}


\subsection{Control of the error}\label{aksjxnjzx}

For any $k\geq 0$, we denote by $\calI(k)$ (resp. $\calU(k)$) the leaves of $\Lambda(k)$ with 
label $\calI$ (resp. $\calU$). We can readily estimate the failure probability~$p$ thanks to the tree~$\Lambda(k)$ 
as follows: 
\[
p^-(k)\leq p \leq p^+(k),
\]
where
\begin{equation}\label{eq:pmpp}
p^-(k):=\sum_{Q\in \calI(k)} \dP(X\in Q) 
\quad\text{and}\quad 
p^+(k):=p^-(k)+\sum_{Q\in \calU(k)} \dP(X\in Q).   
\end{equation}

\begin{Lem}[Control of the error]\label{prop:error}
 For any $k\geq 0$, the estimations $p^-(k)$ and $p^+(k)$ of the failure probability $p$ 
 given by the tree~$\Lambda(k)$ are such that 
 \[
0\leq p^+(k)-p^-(k)\leq CK 2^{-k}.
\]
\end{Lem}

\begin{proof}
Under Assumption \ref{kjncllxnlankx}, for any $k\in\dN$ and $Q\in \calU(k)$, we have 
\[
\lambda(Q)=2^{-dk} 
\quad\text{and}\quad 
\dP(X\in Q)\leq 2^{-dk}K. 
\]
As a consequence, the definition of $p^-(k)$ and $p^+(k)$ together with Equation~\eqref{eq:card-Uj}
ensure that 
$$ p^+(k)-p^-(k)=\sum_{Q\in \calU(k)} \dP(X\in Q)\leq \frac{\ABS{\calU(k)}K}{2^{dk}}\leq CK \frac{2^{k(d-1)}}{2^{dk}}.$$
This concludes the proof.
\end{proof}

Before going further, let us notice that, for any $n\geq 1$, there exists $k\geq 0$ such that $n_k\leq n< n_{k+1}$, with the convention $n_0=0$. We can apply the same algorithm as before, with the understanding that all the leaves of the tree $\Lambda(k)$ are explored while this is the case only for $(n-n_k)$ leaves with depth $(k+1)$ of the tree $\Lambda(k+1)$. This defines a subtree $\Lambda_n$ of $\Lambda(k+1)$. With obvious notation, the leaves of $\Lambda_n$ can be partitioned as $\mathcal{I}_n\cup\mathcal{U}_n\cup\mathcal{O}_n$. In this respect, we deduce upper and lower bounds $p_n^-$ and $p_n^+$ for $p$ as follows:
\begin{equation}
p^-_n:=\sum_{Q\in \calI_n} \dP(X\in Q) 
\quad\text{and}\quad 
p^+_n:=p^-_n+\sum_{Q\in \calU_n} \dP(X\in Q).   
\end{equation}
Clearly, we have
$$p^-(k)\leq p_n^-\leq p\leq p_n^+\leq p^+(k),$$
so that the approximation error $E_n:=p_n^+-p_n^-$ satisfies $E_n\leq p^+(k)-p^-(k)$.\medskip

With this in mind, we can now complete the proof of Theorem \ref{zdjcnlzdc}. When $d\geq 2$, according to Proposition~\ref{prop:nk}, we may write
 \[
n<n_{k+1}\leq 4C\ 2^{(d-1)(k+1)} 
\quad\text{or, equivalently,}\quad 
2^{-k}\leq 2\PAR{\frac{n}{4C}}^{-\frac{1}{d-1}}.
\]
Hence, Lemma \ref{prop:error} yields
\[
E_n\leq p^+(k)-p^-(k)\leq CK 2^{-k}
\leq 8C^{\frac{d}{d-1}}K n^{-\frac{1}{d-1}}.
\]
When $d=1$, the same reasoning gives
 \[
n<n_{k+1}\leq 2C(k+1)
\quad\text{or, equivalently,}\quad 
2^{-k}\leq 2^{1-\frac{n}{2C}},
\]
so that
\[
E_n\leq p^+(k)-p^-(k) \leq 2CK\ 2^{-\frac{n}{2C}}.
\]
This terminates the proof of the first part of Theorem \ref{zdjcnlzdc}. The fact that this error is optimal is shown in Section \ref{dcljanlcna}.

\section{Estimation error}
\label{khckasjbc}

We return to the notation of Section \ref{aksjxnjzx} and recall that, for any $n\geq 2$, we denote by $\calI_n$ (resp. $\calU_n$) the leaves of $\Lambda_n$ with 
label $\calI$ (resp. $\calU$), so that
\[
p^-_n\leq p \leq p^+_n,
\]
where
\begin{equation}\label{eq:pmpp}
p^-_n:=\sum_{Q\in \calI_n} \dP(X\in Q) 
\quad\text{and}\quad 
p^+_n:=p^-_n+\sum_{Q\in \calU_n} \dP(X\in Q).   
\end{equation}
Our goal in this section is to estimate $p^-_n$ and $p^+_n$ with no additional call to $g$. We first do it by assuming that, for each vertex $u\in\Lambda_n$, we can simulate an $N$ i.i.d.\ sample distributed according to the law of $X$ given that it belongs to $Q(u)$. This allows us to propose in Section \ref{lakzlkzx} two idealized estimators $p_{n,N}^-$ and $p_{n,N}^+$ along with their asymptotic variances. In Section \ref{zkjcajclacx}, thanks to MCMC techniques, we construct two estimators $\hat p_{n,N}^-$ and $\hat p_{n,N}^+$ of the latters provided that the density $f_X$ is known up to a normalizing constant.

\subsection{Estimation error in an idealized case}\label{lakzlkzx}

From a given tree $\Lambda_n$, one can estimate the failure probability~$p$ thanks to 
$p^-_n$ and $p^+_n$ defined in Equation~\eqref{eq:pmpp}. To that end, one 
has to compute (or estimate) the probability 
\[
p(u):=\dP(X\in Q(u)),
\]
for each leaf $u$ of $\Lambda_k$. If $u$ is far from the root, then $p(u)$ should be very 
small and difficult to estimate directly through a naive Monte Carlo method, as explained in Section \ref{intro}. Therefore, we propose to apply 
a splitting strategy inspired by rare event estimation.\medskip 

For a given leaf $u\in\Lambda_n$, recall that $a(u)$ stands for the set of 
the ancestors of $u$, including $u$ but excluding the root for convenience. 
Since $\dP(X\in\Omega)=1$, Bayes formula ensures that 
\[
\dP(X\in Q(u))=\prod_{v\in a(u)} \dP(X\in Q(v)\vert X\in Q(\bar v)),
\]
which can be reformulated as follows
\[
p(u)=\prod_{v\in a(u)} q(v) 
\quad\text{where}\quad 
q(v):=\dP(X\in Q(v)\vert X\in Q(\bar v)).
\]

\begin{Ass}[Perfect samplings]\label{ass:perfect}
 Recall that $\calT$ stands for the infinite $2^d$-regular tree. For any $v\in \calT$, consider 
 a sequence ${(X^v_i)}_{i\geq 1}$ of i.i.d.\ random variables with distribution $\calL(X\vert X\in Q(v))$ 
and assume that the sequences ${(X^v)}_{v\in\calT}$ are independent.  
\end{Ass}

\begin{Def}[Ideal estimators]\label{lazjcnlljcjc}
For $N\geq 1$ and $u,v\in\calT$, we define
\begin{equation}\label{eq:def-pchap}
C_N^v:=\sum_{i=1}^N \un_\BRA{X^{\bar v}_i\in Q(v)},
\quad 
q_N(v):=\frac{C^v_N}{N}
\quad\text{and}\quad 
p_N(u):=\prod_{v\in a(u)}q_N(v).
\end{equation}
\end{Def}

\begin{Rem}[Multinomial distribution and unbiasedness]\label{lazclncskal}
The random variable $C_N^v$ is the number of random variables 
${(X^{\bar v}_i)}_{1\leq i\leq N}$ which are in fact in the cube $Q(v)$. Let $w$ be a fixed vertex. 
The distribution of the random vector ${(C^v_N)}_{v\in c(w)}$ 
is the multinomial distribution with parameters $N$ and ${(q(v))}_{v\in c(w)}$. 
As a consequence, $q_N(v)$ is a strongly consistent and unbiased estimator of $q(v)$.
In addition, for  $k$ different vertices $w_1,w_2,\ldots,w_k$ in $\calT$, the vectors 
\[
{(C^v_N)}_{v\in c(w_1)}, {(C^v_N)}_{v\in c(w_2)},\ldots,{(C^v_N)}_{v\in c(w_k)}
\]
are independent. From this we deduce that $p_N(u)$ is also unbiased.  
\end{Rem}

\begin{Rem}
The random variables ${(q_N(v))}_{v\in a(u)}$ are independent. Nevertheless, for two different leaves $u$ and $u'$, $p_N(u)$ and 
$p_N(u')$ are not independent.  
\end{Rem}

Our next result, whose proof is deferred to Section \ref{qlsjcnlnc}, provides asymptotic properties (namely, consistency and asymptotic normality) for the estimator $p_N(\calS)$ of the probability $p(\calS)$ associated to any set of leaves $\calS$. One may keep in mind that, for our problem, we will apply this result with $\calS=\calI_n$ and $\calS=\calI_n\cup\calU_n$, in which case $p(\calS)$ (respectively $p_N(\calS)$) corresponds to $p_n^-$ and $p_n^+$ (respectively $p_{n,N}^-$ and $p_{n,N}^+$).\medskip

\begin{The}\label{th:esti-p}
For any set $\calS$ of leaves of a tree~$\Lambda$, one can estimate 
\[
p(\calS):=\sum_{u\in \calS}p(u)
\quad \text{by} \quad 
p_N(\calS):=\sum_{u\in \calS}p_N(u),
\]
where $p_N(u)$ is defined in \eqref{eq:def-pchap}. The estimator~$p_N(\calS)$ 
is unbiased and strongly consistent:
\[
p_N(\calS) \xrightarrow[N\to\infty]{a.s.} p(\calS).
\]
Moreover, it is asymptotically normal, namely  
\[
\sqrt{N}(p_N(\calS)-p(\calS))\xrightarrow[N\to\infty]{\mathcal{D}}\calN(0,\sigma^2),
\]
where
 \[
\sigma^2=\sum_{u\in\calS}p(u)^2\sum_{v\in a(u)}\frac{1-q(v)}{q(v)}
+\sum_{\substack{u,u'\in\calS\\ u\neq u'}}p(u)p(u')
\SBRA{\sum_{v\in a(u\wedge u')}\frac{1-q(v)}{q(v)}-1}.
\]
\end{The}

Remark that if $q(v)=0$ then $p(u)=0$ whenever $v\leq u$, so that one can cancel $u$ from the set of leaves $\calS$ and the expression of $\sigma^2$ is always well-defined.

\begin{Rem}[Variance estimation]\label{kjazxjcx}
Recall that each $q(v)$ is strictly positive and consistently estimated on the fly by $q_N(v)$, so that 
$\sigma^2$ is readily estimated by 
\[
\sigma_N^2=\sum_{u\in\calS}p_N(u)^2\sum_{v\in a(u)}\frac{1-q_N(v)}{q_N(v)}
+\sum_{\substack{u,u'\in\calS\\ u\neq u'}}p_N(u)p_N(u')
\SBRA{\sum_{v\in a(u\wedge u')}\frac{1-q_N(v)}{q_N(v)}-1}
\]
and $\sigma_N^2$ goes almost surely to $\sigma^2$ when $N$ goes to infinity. Hence, Slutsky's lemma ensures that
\[
\sqrt{N}\ \frac{p_N(\calS)-p(\calS)}{\sigma_N}\xrightarrow[N\to\infty]{d}\calN(0,1).
\]
In particular, the latter provides asymptotic confidence intervals for $p(\calS)$.
\end{Rem}

For our concern, recall that $p^-_n\leq p \leq p^+_n$ where
\[
p^-_n=p(\calI_n)=\sum_{u\in \calI_n} p(u) 
\quad\text{and}\quad 
p^+_n=p(\calI_n\cup\calU_n). 
\]
Hence, the sets of leaves of interest are $\calS=\calI_n$ and $\calS=\calI_n\cup\calU_n$. Indeed, the previous results establish that
\[
p_{n,N}^-=p_N(\calI_n)\xrightarrow[N\to\infty]{a.s.}p^-_n\leq p\leq p^+_n\xleftarrow[N\to\infty]{a.s.}p_{N}(\calI_n\cup\calU_n)=p_{n,N}^+,
\]
as well as
$$\sqrt{N}\left(p_{n,N}^\pm-p_{n}^\pm\right)\xrightarrow[N\to\infty]{d}\mathcal{N}(0,(\sigma_n^\pm)^2).$$
In addition, by Remark \ref{kjazxjcx}, we can construct on the fly consistent estimators $\sigma_{n,N}^-$ and $\sigma_{n,N}^+$ of the latter asymptotic standard deviations. This closes the proof of Theorem \ref{amicjpazijcpaic}.\medskip

\begin{Rem}[Asymptotic confidence intervals]\label{aiclazcalzd}
Denote by $\Phi$ the cumulative distribution function of the standard normal distribution so that, for $\alpha\in(0,1)$, $\Phi^{-1}(1-\alpha/2)$ is the $(1-\alpha/2)$ quantile. If we define 
$$m_{n,N}:=p_{n,N}^- -\frac{\Phi^{-1}(1-\alpha/2)\sigma_{n,N}^-}{\sqrt{N}}$$
as well as
$$ M_{n,N}:=p_{n,N}^+ +\frac{\Phi^{-1}(1-\alpha/2)\sigma_{n,N}^+}{\sqrt{N}},$$
then $[m_{n,N},1]$ and $[0,M_{n,N}]$ are $100(1-\alpha/2)\%$ asymptotic confidence intervals for, respectively, $p_n^-$ and $p_n^+$. Since $p_n^-\leq p\leq p_n^+$, the union bound ensures that $[m_{n,N},M_{n,N}]$ is a $100(1-\alpha)\%$ asymptotic confidence interval for $p$. 
\end{Rem}



\subsection{Estimation error in practice}\label{zkjcajclacx}

The purpose of this section is to prove Proposition \ref{peidpiejcpzejc}. Recall from Definition \ref{lazjcnlljcjc} that each leaf probability 
\[
p(u)=\dP(X\in Q(u))=\prod_{v\in a(u)}\dP(X\in Q(v)\vert X\in Q(\bar v))=\prod_{v\in a(u)}q(v)
\]
 is estimated by 
\[
p_N(u)=\prod_{v\in a(u)}q_N(v)\quad\text{where}\quad q_N(v)=\frac{C^v_N}{N}=\frac{1}{N}\sum_{i=1}^N \un_\BRA{X^{\bar v}_i\in Q(v)}.
\]
To apply the results of Theorem \ref{th:esti-p}, this supposes that, for each $\bar v$, we have a sample of $N$ i.i.d.\ random variables $X^{\bar v}_i$. In addition,  for two vertices $v$ and $v'$ such that ${\bar v}\neq{\bar v'}$, these samples must be independent. The present section explains how to reach this goal, at least approximately.\medskip

Consider a fixed vertex $\bar v$, denote $\mu_{\bar v}=\mathcal{L}(X|X\in Q(\bar v))$, and $f_{\bar v}$ the corresponding probability density function, that is
$$\mu_{\bar v}(dx)=f_{\bar v}(x)dx=\frac{1}{\P(X\in Q(\bar v))}f_X(x)\un_{x\in Q(\bar v)}.$$
Starting from a point $X_0\sim\mathcal{U}_{\bar v}$ the uniform law on $Q(\bar v)$, the Metropolis-Hastings algorithm allows us to construct a Markov chain $(X_n)$ with asymptotic distribution $\mu_{\bar v}$.\medskip

We refer the interested reader to Tierney \cite{MR1329166} for a thorough presentation as well as numerous theoretical results on Markov chain Monte Carlo methods. For our purpose, we just present the idea for a specific choice of the Markov dynamics, which turns out to be a particular case of independent Metropolis.\medskip

 Here is the mechanism: starting from $X_t$, simulate $X'_{t}\sim\mathcal{U}_{\bar v}$ and set
\begin{equation}\label{jcbkabckabc}
X_{t+1}:=X'_t\un_{U_{t+1}\leq f_X(X'_t)/f_X(X_t)}+X_t\un_{U_{t+1}> f_X(X'_t)/f_X(X_t)},
\end{equation}
where $(U_t)_{t\in\N^\star}$ is a sequence of i.i.d.\ random variables with uniform law on $[0,1]$. Needless to say, in the previous expression,  $X_t$, $X'_{t}$, and $U_{t+1}$ are also assumed independent. It is readily seen that, if we denote by $K_{\bar v}$ the transition kernel associated to this Markov chain, then $K_{\bar v}$ is $\mu_{\bar v}$-reversible so that, under appropriate assumptions, $(X_t)$ goes in distribution to $\mu_{\bar v}$.\medskip

In order to make this convergence more precise, let us recall that the total variation distance between two probability measures $\mu$ and $\nu$ on $Q(\bar v)$ is 
$$\|\mu-\nu\|_{TV}:=\sup_{B\in\mathcal{B}_{\bar v}}|\mu(B)-\nu(B)|,$$
where $\mathcal{B}_{\bar v}$ is the collection of all Borel sets on $Q(\bar v)$. Denoting $\delta_x$ the Dirac measure at $x$ and $\delta_xK_{\bar v}^t$ the law of $X_t$ for the above Markov chain with initial condition $X_0=x$, we say that the chain is uniformly ergodic on $Q(\bar v)$ if there exist $A_{\bar v}>0$ and $0<r_{\bar v}<1$ such that, for all $t\in\N^\star$, 
$$\sup_{x\in Q(\bar v)}\|\delta_xK_{\bar v}^t-\mu_{\bar v}\|_{TV}\leq A_{\bar v} r_{\bar v}^t.$$ 
Let $g_{\bar v}$ stand for the density of the uniform distribution on $Q(\bar v)$ and
\begin{equation}\label{ajcac} 
\beta_{\bar v}^{-1}:=\sup_{x\in Q(\bar v)}\frac{f_{\bar v}(x)}{g_{\bar v}(x)},
\end{equation}
then Corollary 4 in \cite{MR1329166} ensures that the Markov chain $(X_t)$ is uniformly ergodic with convergence rate $r_{\bar v}\leq1-\beta_{\bar v}$. In our context, notice that the latter is always strictly less than 1 if, for example, $f_X$ is continuous and strictly positive on $\Omega=[0,1]^d$, hence our assumption in Proposition \ref{peidpiejcpzejc}.\medskip

To see the consequence of this result in our context, remember the coupling interpretation of the total variation distance, that is
$$\|\mu-\nu\|_{TV}=\inf_{(X,Y)}\P(X\neq Y),$$
where the infimum is over all couples of random variables on $Q(\bar v)\times Q(\bar v)$ with marginal laws $\mu$ and $\nu$. More precisely, given $X$ with law $\nu$, it is always possible to construct a random variable $Y$ with law $\mu$ such that the equality is achieved, i.e., $\P(X\neq Y)=\|\mu-\nu\|_{TV}$.\medskip

Hence, if we consider as above a Markov chain $(X_t^{\bar v})$ with arbitrary initial condition, for example $X_0^{\bar v}$ with uniform law $\mathcal{U}_{\bar v}$ on $Q(\bar v)$, there exists a random variable $X_\infty^{\bar v}$ with law $\mu_{\bar v}$ such that 
$$\P(X_t^{\bar v}\neq X_\infty^{\bar v})=\|\mathcal{U}_{\bar v}K_{\bar v}^t-\mu_{\bar v}\|_{TV}\leq A_{\bar v} r_{\bar v}^t.$$
Therefore, if we start from $N$ i.i.d.\ initial conditions $\mathbf{X}_0^{\bar v}:=(X_0^{\bar v,(1)},\dots,X_0^{\bar v,(N)})$ with uniform distribution on $Q(\bar v)$, and run independently during $t$ steps the previous Metropolis algorithm to obtain the sample $\mathbf{X}_t^{\bar v}:=(X_t^{\bar v,(1)},\dots,X_t^{\bar v,(N)})$, we deduce that
$$\P(\mathbf{X}_t^{\bar v}=\mathbf{X}_\infty^{\bar v})\geq (1-A_{\bar v} r_{\bar v}^t)^N,$$
where $\mathbf{X}_\infty^{\bar v}:=(X_\infty^{\bar v,(1)},\dots,X_\infty^{\bar v,(N)})\sim\mu_{\bar v}^{\otimes n}$.\medskip

Next, apply the previous procedure to each vertex $\bar v$ of the considered tree $\Lambda$, denote by $\mathcal{X}_t:=(\mathbf{X}_t^{\bar v})_{\bar v\in\Lambda}$ all the corresponding sets of $N$ i.i.d.\ samples, and $\mathcal{X}_\infty:=(\mathbf{X}_\infty^{\bar v})_{\bar v\in\Lambda}$ the corresponding sets of $N$ i.i.d.\ ``idealized'' samples. Denoting $A_{\Lambda}:=\max_{\bar v\in\Lambda}A_{\bar v}$ and $r_{\Lambda}:=\max_{\bar v\in\Lambda}r_{\bar v}\in(0,1)$, we deduce that
$$\P(\mathcal{X}_t=\mathcal{X}_\infty)\geq (1-A_{\Lambda} r_{\Lambda}^t)^{|\Lambda|N}.$$
For each vertex $v$ and each leaf $u$, consider the estimators
\[
\hat p_N(u):=\prod_{v\in a(u)}\hat q_N(v)\quad\text{where}\quad \hat q_N(v):=\frac{1}{N}\sum_{i=1}^N \un_\BRA{X^{\bar v,(i)}_t\in Q(v)},
\]
and, for any set $\calS$ of leaves of the tree~$\Lambda$,
$$\hat p_N(\calS):=\sum_{u\in \calS}\hat p_N(u).$$
Clearly, on the event $\{\mathcal{X}_t=\mathcal{X}_\infty\}$, we have $\hat p_N(\calS)= p_N(\calS)$, which means that
$$\P(\hat p_N(\calS)= p_N(\calS))\geq (1-A_{\Lambda} r_{\Lambda}^t)^{|\Lambda|N},$$
 where $p_N(\calS)$ is the ideal estimator defined in Theorem \ref{th:esti-p}. Finally, it suffices to consider $\calS=\mathcal{I}_n$ and $\calS=\mathcal{I}_n\cup\mathcal{U}_n$ to conclude the proof of Proposition \ref{peidpiejcpzejc}.
 
 \begin{Rem}[Confidence intervals in practice]\label{aiclazcalzdbis}
Mutatis mutandis, the result of Remark \ref{aiclazcalzd} is still valid. Specifically, if we denote 
$$\hat m_{n,N}:=\hat p_{n,N}^- -\frac{\Phi^{-1}(1-\alpha/2)\hat \sigma_{n,N}^-}{\sqrt{N}}$$
as well as
$$\hat M_{n,N}:=\hat p_{n,N}^+ +\frac{\Phi^{-1}(1-\alpha/2)\hat \sigma_{n,N}^+}{\sqrt{N}},$$
then on the event $\{\mathcal{X}_t=\mathcal{X}_\infty\}$, $[\hat m_{n,N},\hat M_{n,N}]$ is a $100(1-\alpha)\%$ asymptotic confidence interval for $p$. This will be illustrated in Section \ref{ksjbksjcxksj}.
\end{Rem}

\begin{Rem}
Returning to \eqref{ajcac}, one can notice that the smaller the side length of $Q(\bar v)$, the faster the convergence of the Metropolis algorithm. Indeed, denoting $c_{Q(\bar v)}$ its center and $\lambda(Q(\bar v))$ its Lebesgue measure, the continuity of $f_X$ ensures that, when $\lambda(Q(\bar v))\to 0$,
$$f_{\bar v}(x)=\frac{f_X(x)\un_{x\in Q(\bar v)}}{\P(X\in Q(\bar v))}\approx\frac{f_X(c_{Q(\bar v)})\un_{x\in Q(\bar v)}}{f_X(c_{Q(\bar v)})\lambda(Q(\bar v))}=g_{\bar v}(x),$$
which means that $\beta_{\bar v}$ goes to 1 or, equivalently, that $r_{\bar v}$ goes to 0.
\end{Rem}

\section{Numerical illustration}\label{ksjbksjcxksj}



\begin{figure}[h]
\begin{minipage}[b]{1\linewidth}
\begin{center}
\includegraphics[width=6.25cm, height = 4.5cm]{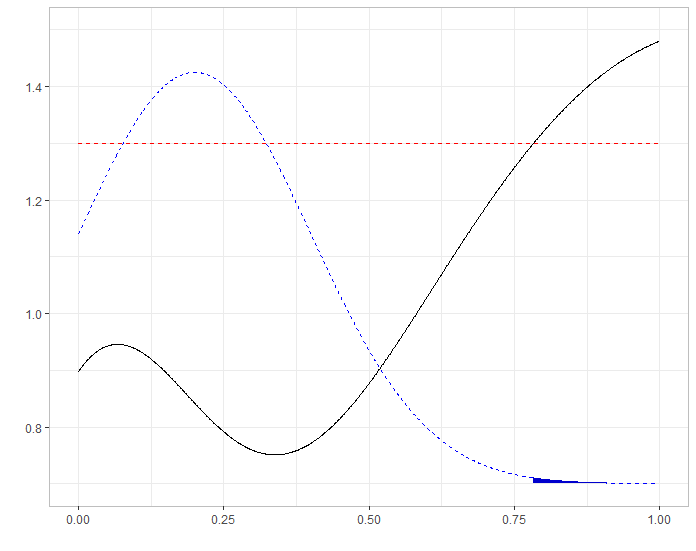}
\includegraphics[width=6.25cm, height = 4.5cm]{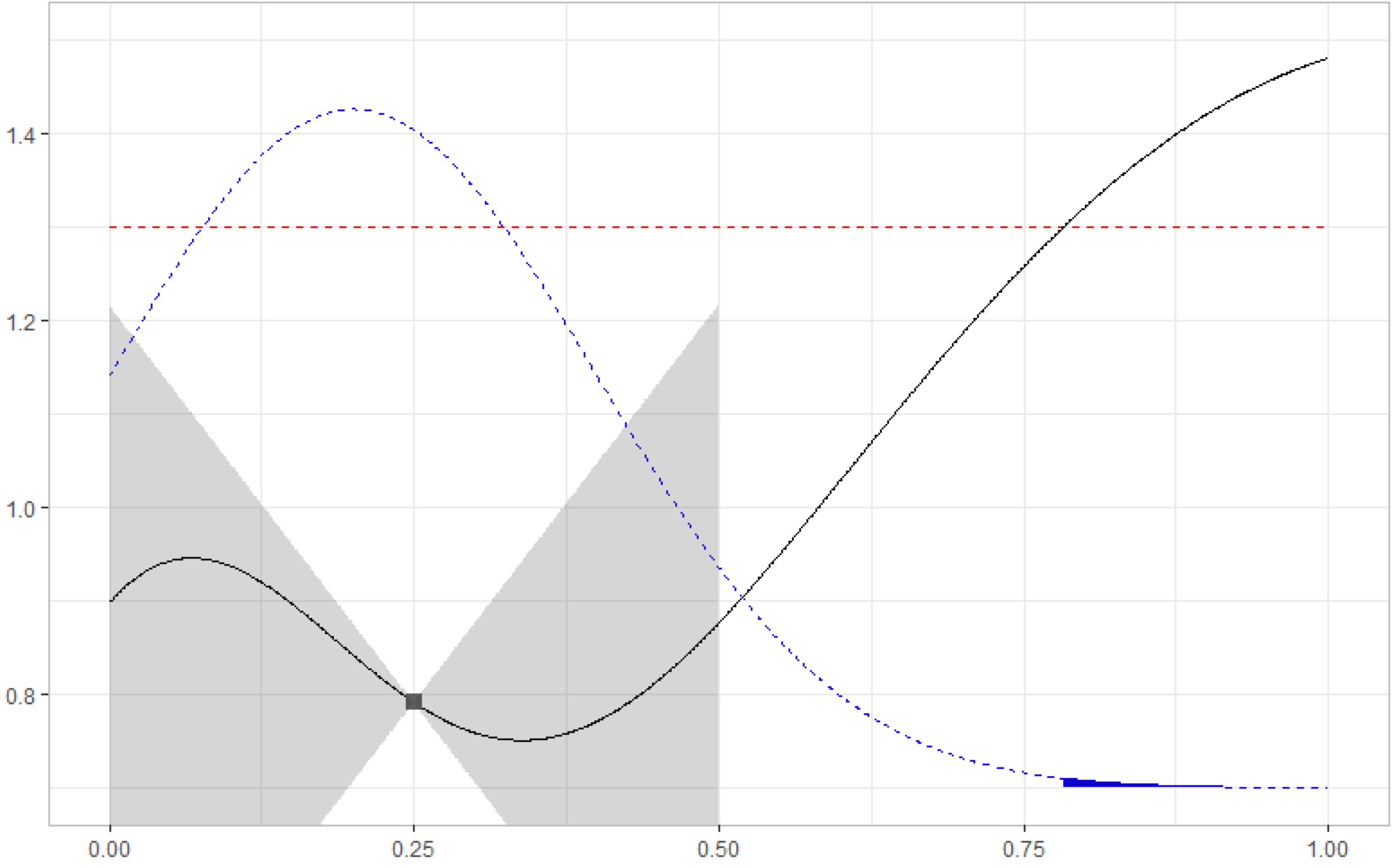}\\
  \vspace{0pt} \includegraphics[width=6.25cm, height = 4.5cm]{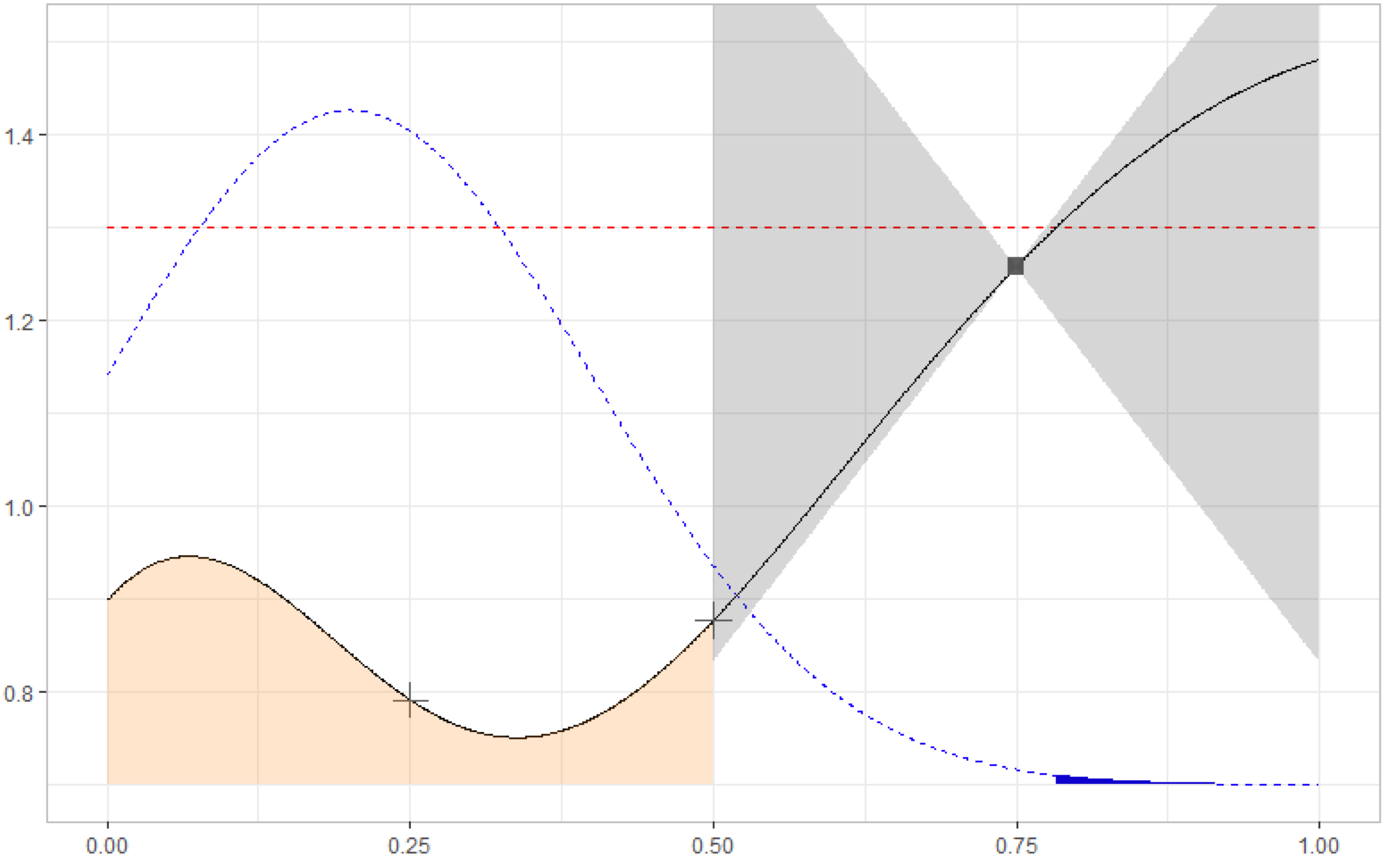} 
  \vspace{0pt} \includegraphics[width=6.25cm, height = 4.5cm]{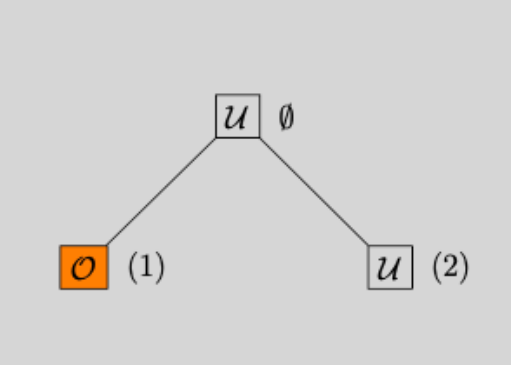}	
\end{center}
\end{minipage}
\caption{Representation of the function $g$ (black), the pdf $f_X$ (blue), the threshold $T$ (red), the probability $p$ (blue region), and illustration of the first step of the algorithm.}
\label{ljkcnlakenclkae}
\end{figure}

To illustrate our algorithm, we consider a toy example which is just a variant of the one proposed in Section 5.1 of \cite{MR2909621}. For all $x\in[0,1]$, we set
$$g(x) = (0.8x - 0.3) + \exp\left(-11.534x^{1.95}\right)+\exp\left(-2(x-0.9)^2\right),$$
which is $L$-Lipschitz with $L = \sup_{x\in[0,1]}|g'(x)|\approx1.61$. The law of $X$ is the restriction of a Gaussian distribution $\mathcal{N}(1/5,1/25)$ to $[0,1]$, i.e.,
$$f_X(x)\propto\exp\left\{-\frac{25}{2}\left(x-\frac{1}{5}\right)^2\right\}\mathbf{1}_{[0,1]}(x).$$
Finally, we take $T = 1.3$, so that a standard numerical integration gives $p\approx2.08\times 10^{-3}$. This is illustrated on Figure \ref{ljkcnlakenclkae}, together with the first step of the algorithm. Recall that the evaluation of $g$ at point $x=1/2$ is useless. Indeed, since $0<p<1$, the interval $\Omega=[0,1]$ is necessarily classified as uncertain (i.e., $\mathcal{U}$). Therefore, the first step consists in computing $g(1/4)$ and $g(3/4)$, which correspond respectively to vertices $(1)$ and $(2)$ of the tree. From this figure, it is easy to see that $(1)$ is classified as out (i.e., $\mathcal{O}$) while $(2)$ is classified as uncertain (i.e., $\mathcal{U}$). Therefore, there is no need to further investigate the interval $[0,1/2]$.\medskip

\begin{figure}[h]
\begin{minipage}[b]{1\linewidth}
\begin{center}
\includegraphics[width=6.25cm, height = 4.5cm]{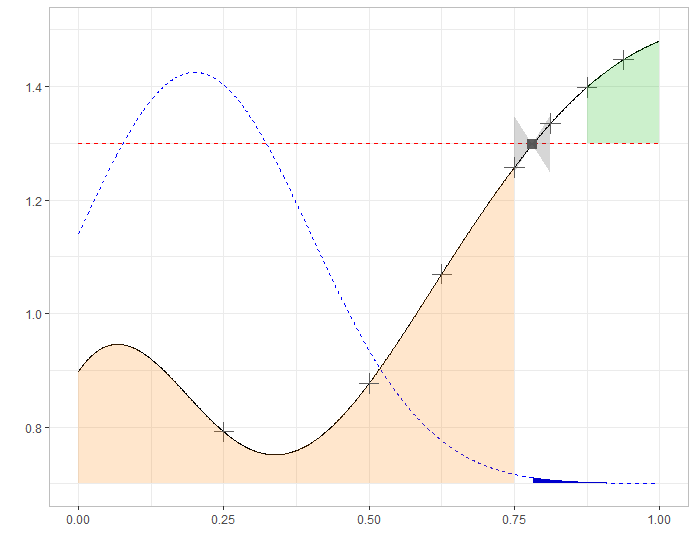}
\includegraphics[width=6.25cm, height = 4.5cm]{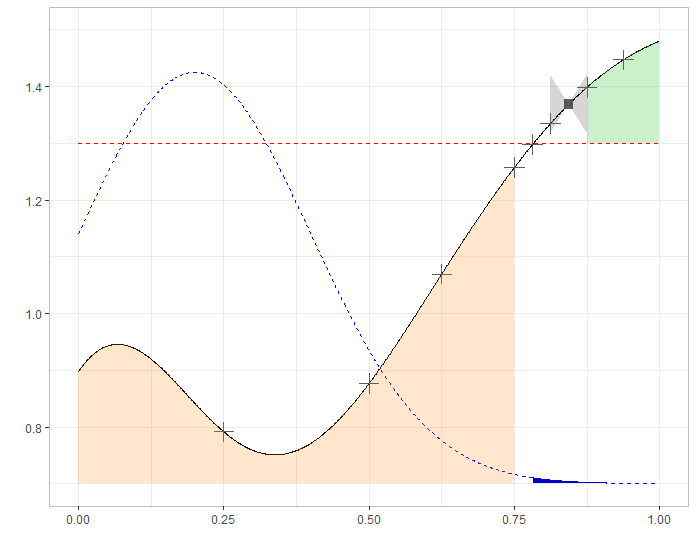}\\
\vspace{0pt} \includegraphics[width=6.25cm, height = 4.5cm]{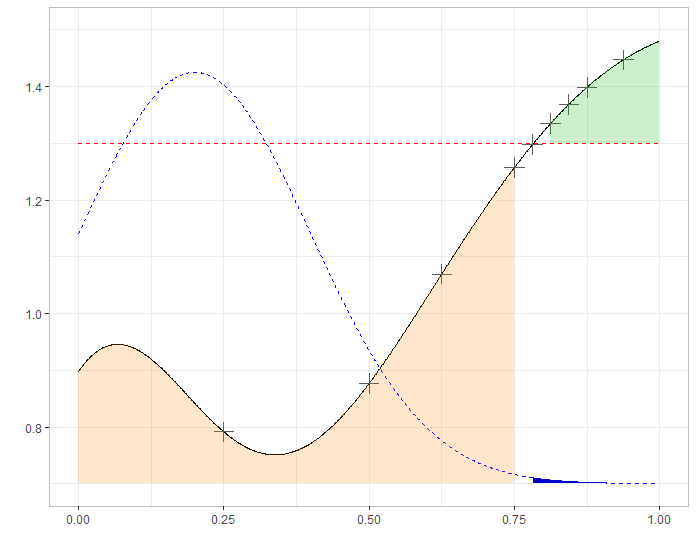} 
\vspace{0pt} \includegraphics[width=6.25cm, height = 4.5cm]{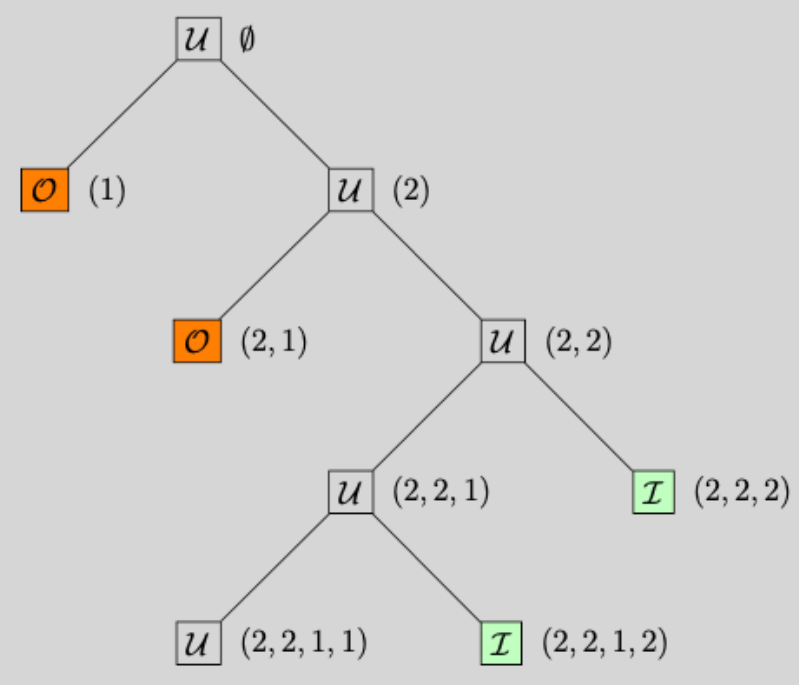}
\end{center}
\end{minipage}
\caption{Step 4 of the algorithm.}
\label{alechazjhlen}
\end{figure}


Figure \ref{alechazjhlen} represents step 4 of the algorithm, which consists in evaluating $g$ at points $x=25/32$ (i.e., vertex $(2,2,1,1)$) and $x=27/32$ (i.e., vertex $(2,2,1,2)$). These evaluations lead to classify the interval $[24/32,26/32]$ as uncertain (i.e., $\mathcal{U}$) and the interval  $[26/32,28/32]$ as included in the failure domain (i.e., $\mathcal{I}$). At this  point, the deterministic lower and upper bounds for $p$ are thus 
$$p^-(4)=\P(X\in[13/16,1])\approx 1.3\times 10^{-3}\leq p,$$
and
$$p\leq p^+(4)=\P(X\in[12/16,1])\approx 2.5\times 10^{-3},$$
and the approximation error is simply 
$$p^+(4)-p^-(4)=\P(X\in[12/16,13/16])\approx 2.2\times 10^{-3}.$$
Unsurprisingly, one may notice that the upper bound given by Lemma \ref{prop:error} is very pessimistic. Indeed, since $d=1$ we know that $C\geq 1$ (see Section \ref{kdjcnkajsc}) and this upper bound can be minorized as follows:
$$CK 2^{-k}\geq \frac{2^{-4}}{\displaystyle\int_0^1 \exp\left\{-\frac{25}{2}\left(x-\frac{1}{5}\right)^2\right\}dx}\approx 0.148\gg 2.2\times 10^{-3}.$$

\begin{figure}
\begin{center}
\begin{minipage}[b]{1\linewidth} 
\includegraphics[width=6.8cm, height = 5.5cm]{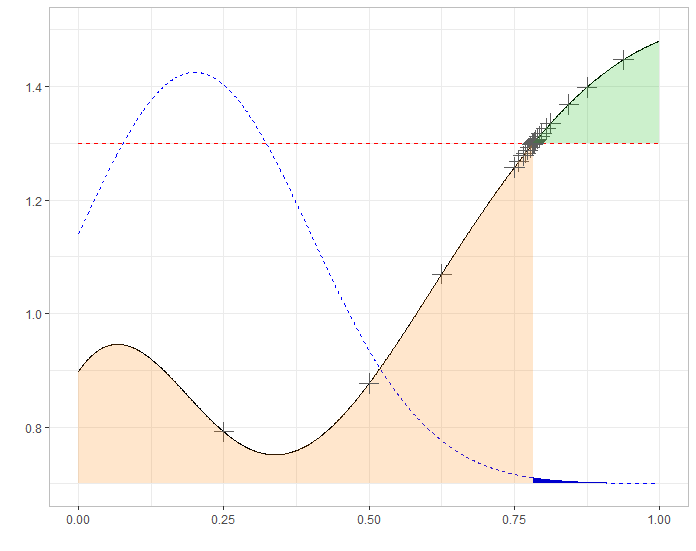}
\includegraphics[width=6.8cm, height =5.5cm]{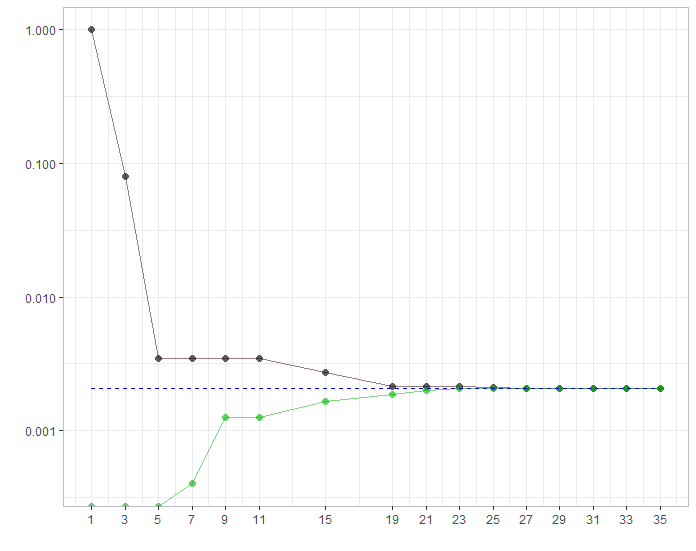}
\end{minipage}
\caption{Convergence of $p_n^-$ and $p_n^+$.}
\label{laclaclak}
\end{center}
\end{figure}

On this toy example, since the law of $X$ is simply the restriction of a Gaussian distribution, it is easy to have a very precise numerical approximation of $\P(X\in Q)$ for any dyadic interval $Q$ and, in turn, for the lower and upper bounds at each step of the algorithm. In other words, we can easily compute the (deterministic) approximation error. The evolution of these bounds $p_n^-$ and $p_n^+$ as the number of evaluation points grows is given in Figure \ref{laclaclak} for a total budget of $n=35$ calls to $g$.\medskip

However, in practice, this is usually not possible, hence the use of MCMC techniques as explained in Section \ref{zkjcajclacx}. On our example, up to a normalizing constant, the pdf $f_X$ is defined by
$$f_X(x)\propto\exp\left\{-\frac{25}{2}\left(x-\frac{1}{5}\right)^2\right\}\mathbf{1}_{[0,1]}(x).$$
Thus, for any couple of points $(x, x')$, the Metropolis ratio $f_X(x')/f_X(x)$ that appears in \eqref{jcbkabckabc} is very easy to compute. We have applied this idea for a sample size $N=10^5$ with $t=25$ Markov transitions for each probability estimation. In this respect, Figure \ref{ecaecaec} shows that when $N$ is much larger than the approximation error, then the latter is much larger than the estimation error. In order to illustrate Remark \ref{aiclazcalzd}, the asymptotic confidence intervals are also given.

\begin{figure}[h]
\begin{center}
\begin{minipage}[b]{1\linewidth} 
\includegraphics[width=6.8cm, height = 5.5cm]{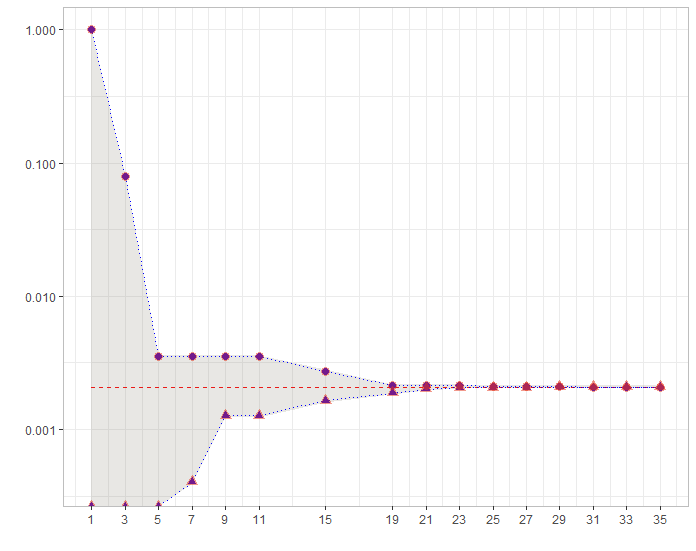}
\includegraphics[width=6.8cm, height =5.5cm]{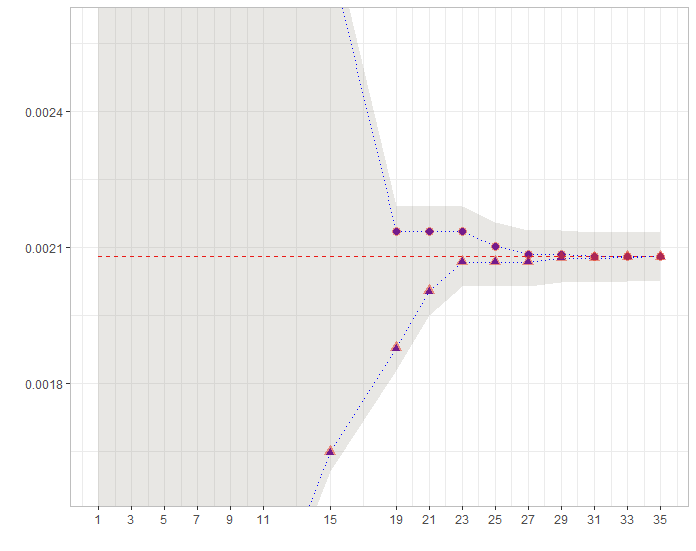}
\end{minipage}
\caption{Estimators $\hat p_{n,N}^-$ and $\hat p_{n,N}^+$ of $p_{n}^-$ and $p_{n}^+$, together with asymptotic confidence intervals.}
\label{ecaecaec}
\end{center}
\end{figure}

\section{Optimality}\label{dcljanlcna}
We have established in Theorem \ref{zdjcnlzdc} that after $n$ evaluations of the function $g$, 
the approximation error of our algorithm is of polynomial order $n^{-\frac{1}{d-1}}$ when $d\geq 2$, and of exponential order $2^{-\beta n}$ when $d=1$. The aim of this section is to show that these bounds are optimal, meaning that they cannot be improved by any other algorithm under the sole general assumptions that we have
made on the function $g$.

\subsection{The case $d\geq 2$} 

When $d\geq 2$, we consider the following particular case:
\begin{itemize}
 \item The random variable $X$ is uniformly distributed on $\Omega=[0,1]^d$;
 \item The function $g$ is defined by $g(x)=-x^1$ for $x=(x^1,\ldots,x^d)\in \Omega$;
 \item The threshold $T$ is equal to $0$. 
\end{itemize}
Thus, in this setting, the failure probability $p:=\dP(g(X)>0)$ is equal to $0$. Clearly the function $g$ satisfies the Lipschitz Assumption \ref{kdjcnkajsc}  with $L=1$ and the level set Assumption \ref{hyp:level-set} with $M=1$.\medskip 

Let us fix an integer $n=2^{j(d-1)-1}$ for some $j\in\dN^\star$ and $n$ arbitrary points $x_1,\ldots,x_n$ in $\Omega$.   
In the sequel, we construct a function $\tilde g$ on $\Omega$ such that 
\begin{itemize}
\item $g(x_i)=\tilde g(x_i)$ for $1\leq i\leq n$;
\item $\tilde p:=\dP(\tilde g(X)>0)\geq c n^{-\frac{1}{d-1}}$;  
\item $\tilde g$ satisfies Assumptions 2 and 3
with $L$ and $M$ independent of $n$.
\end{itemize}
The first fact ensures that any algorithm based on the points ${(x_i)}_{1\leq i\leq n}$ 
leads to the same estimation for $p$ and $\tilde p$. The second one ensures that $\tilde p-p$ 
is (at least) of order $n^{-\frac{1}{d-1}}$.\medskip

First, let us define the face
\[
\mathcal{C}:=\{x\in \Omega\,:\, g(x)=0\}=\left\{x\in\Omega\,:\, x^{1}=0\right\}. 
\]
Consider the set 
\[
\mathcal{D}^\star_j=\{Q\in\mathcal{D}_j\,:\, \mathrm{dist}(Q,\mathcal{C})=0\}
\]
of dyadic cubes with side length $2^{-j}$ which intersect $\mathcal{C}$ and the set 
\[
\tilde{\mathcal{D}}_j=\{Q\in\mathcal{D}^\star_j\,:\, \ x_i\notin Q, \; i=1,\dots,n\}
\]
of dyadic cubes that intersect $\mathcal{C}$ and do not contain any point~$x_i$. 
Since the cardinal of $\mathcal{D}^\star_j$ is equal to $2^{j(d-1)}=2n$, 
the cardinal of $\tilde{\mathcal{D}}_j$ is at least $n$.\medskip 

Second, for any cube $Q$ and any $x\in\Omega$, let us introduce the piecewise affine function 
\[
h_Q(x)= \mathrm{dist}(x,Q^c)=\inf_{y\in Q^c}\NRM{x-y}_\infty
\]
where $Q^c:=\Omega\setminus Q$. The function $h_Q$ is
thus supported on $Q$ and it is $1$-Lipschitz
for the $\ell^\infty$ norm. \medskip

Finally, consider the function $\tilde g$ defined 
as follows on $\Omega$: 
\[
\tilde g=g+2 \sum_{Q\in \tilde{\mathcal{D}}_j}h_Q. 
\]
By construction, the functions $g$ and $\tilde g$ coincide on the cubes that 
do not belong to $\tilde{\mathcal{D}}_j$. In particular, $g(x_i)=\tilde g(x_i)$ for any $1\leq i\leq n$. \medskip

Additionally, since $\sum_{Q\in \tilde{\mathcal{D}}_j}h_Q$ is $1$-Lipschitz, the function $\tilde g$ is $3$-Lipschitz, and therefore
Assumption \ref{kdjcnkajsc} holds with $L=3$.\medskip
 
For any $Q\in \tilde{\mathcal{D}}_j$, if $x_Q$ is the center
of $Q$ one has 
$$
\tilde g(x_Q)=g(x_Q)+2h_Q(x_Q)=-2^{-j-1}+2^{-j}=2^{-j-1}.
$$
Therefore $\tilde g(x)>0$ for any $x\in Q$ such that
$\|x-x_Q\|_\infty\leq \frac {2^{-j-1}}{3}$. As a consequence, since $X$ has a uniform distribution on $\Omega$, the failure probability associated to $\tilde g$ satisfies
$$
\dP(\tilde g(X)>0) \geq n 3^{-d}2^{-dj}
=c n^{-\frac{1}{d-1}},  
$$
where $c=3^{-d}2^{-\frac{d}{d-1}}$.
 \medskip

Finally, let us prove the validity of the level set Assumption \ref{hyp:level-set} for the function $\tilde g$. Just like $g$, the absolute 
value of $\tilde g$ is smaller than $2^{-j}$ on 
the cubes $Q\in \mathcal{D}^\star_j$ and larger elsewhere.
Therefore, when $\delta\geq 2^{-j}$, it is readily seen that
\[
\lambda\PAR{\BRA{x\in \Omega \,:\,|\tilde g(x)|\leq \delta}}\leq \delta.
\]
For the values $\delta\leq 2^{-j}$, we know that 
$\BRA{x\in \Omega \,:\,|\tilde g(x)|\leq \delta}$ is contained in
the union of the cubes $Q\in \mathcal{D}^\star_j$. If $Q\notin  \tilde{\mathcal{D}}_j$, then
\[
\lambda\PAR{\BRA{x\in Q \,:\,|\tilde g(x)|\leq \delta}}\leq \delta 2^{-j(d-1)}.
\]
The cubes $Q\in \tilde{\mathcal{D}}_j$ are treated by 
noticing that on such a cube, the function $\tilde g(x)=-x^1+2h_Q(x)$ is a rescaled version of the function 
$g^*(x)=-x^1+2h_\Omega(x)$ defined on $\Omega$. The gradient
of this function is piecewise constant with $\|\nabla g^*(x)\|_1\geq 1$
and therefore $|\nabla g^*(x)|\geq \frac 1 {\sqrt d}$ almost
everywhere on $\Omega$. In addition $g^*$ vanishes
on a polyhedral shaped set $S$ of $(d-1)$-dimensional measure $1<H<\infty$ since in particular $g^*(x)=0$ if $x^1=0$. Using the coarea formula \eqref{coarea}, this yields 
\[
\lambda\PAR{\BRA{x\in \Omega \,:\,|g^*(x)|\leq \delta}}\leq 2\sqrt{d}H\delta,
\]
for $\delta>0$ small enough, and therefore 
\[
\lambda\PAR{\BRA{x\in \Omega \,:\,|g^*(x)|\leq \delta}}\leq B\delta,
\]
for all value of $\delta>0$ up to possibly taking a constant $B$ larger than $2\sqrt{d}H$.
By rescaling
\[
\lambda\PAR{\BRA{x\in Q \,:\,|\tilde g(x)|\leq \delta}}\leq B\delta 2^{-j(d-1)}.
\]
for all $\delta\leq 2^{-j}$. Summing on all $Q\in \mathcal{D}^\star_j$, since $H>1$ and  $|\mathcal{D}^\star_j|=2n$, we find that
\[
\lambda\PAR{\BRA{x\in \Omega \,:\,|\tilde g(x)|\leq \delta}}\leq B\delta.
\]
This shows that Assumption \ref{hyp:level-set} holds with $M=B$ independent of $n$.\medskip

This proves the optimality of the approximation error rate of our algorithm.

\subsection{The case $d=1$}

The idea is the same as for the case $d\geq 2$. More precisely, we consider the following setting:
\begin{itemize}
 \item The random variable $X$ is uniformly distributed on $\Omega=[0,1]$;
 \item The function $g$ is defined by $g(x)=-x$;
 \item The threshold $T$ is equal to $0$. 
\end{itemize}
As in the previous subsection, the failure probability $p:=\dP(g(X)>0)$ is thus equal to $0$, and
the function $g$ satisfies Assumption \ref{kdjcnkajsc}  with $L=1$ and Assumption \ref{hyp:level-set} with $M=1$.\medskip 
\medskip 

Let us fix an integer $n\in\dN^\star$ and $n$ points $x_1,\ldots,x_n$ in $\Omega$.   
As before, the idea is to construct a function $\tilde g$ on $\Omega$ such that 
\begin{itemize}
\item $g(x_i)=\tilde g(x_i)$ for $1\leq i\leq n$;
\item $\tilde p:=\dP(\tilde g(X)>0)\geq c2^{-n}$.  
\item $\tilde g$ satisfies Assumptions 2 and 3
with $L$ and $M$ independent of $n$.
\end{itemize}

First, we define $I_{n+1}:=[0,2^{-n}]$ and, for $1\leq j\leq n$, $I_j:=[2^{-j},2^{-(j-1)}]$. To mimic the previous notation, this set of $(n+1)$ intervals is denoted $\mathcal{D}^\star_n$ and, accordingly,
\[
\tilde{\mathcal{D}}_n=\{I\in\mathcal{D}^\star_n\,:\, \forall i=1,\ldots,n,\ x_i\notin I\}
\]
stands for the set of intervals that do not contain any point~$x_i$. 
Since the cardinal of $\mathcal{D}^\star_n$ is equal to $(n+1)$, 
the cardinal of $\tilde{\mathcal{D}}_n$ is at least equal to $1$.\medskip 

Second, for any interval $I$ and any $x\in\Omega$, we consider the $1$-Lipschitz function 
\[
h_I(x)= \mathrm{dist}(x,I^c)=\inf_{y\in I^c}|x-y|.
\]
Finally, we pick one interval $J\in \tilde{\mathcal{D}}_n$
and define the function $\tilde g$ defined 
as follows
\[
\tilde g=g+4h_J.
\]
As before, the functions $g$ and $\tilde g$ coincide on $\Omega\setminus J$. In particular, $g(x_i)=\tilde g(x_i)$ for any $1\leq i\leq n$. \medskip

Additionally, the function $\tilde g$ is $5$-Lipschitz, and therefore Assumption \ref{kdjcnkajsc} holds with $L=5$. Since $\tilde g$ vanishes
at $x=0$ and (at most) at two other points inside $J$ where its gradient is larger than $3$, it is also easily seen that Assumption \ref{hyp:level-set} holds with $M=7/3$.
\medskip
  
If $x_J$ denotes the center of $J$, then one has
$$
\tilde g(x_J)=g(x_J)+4h_J(x_J)=-\frac 3 4  2^{-(j-1)}+4\times2^{-(j+1)}
=2^{-(j+1)}>0,
$$
in the case $J=I_j=[2^{-j},2^{-(j-1)}]$, $1\leq j\leq n$, and 
$$
\tilde g(x_J)=g(x_J)+4h_J(x_J)=-\frac 1 2 2^{-n}+4\times2^{-(n+1)}
=3\times 2^{-(n+1)}>0,
$$
in the case $J=I_{n+1}=[0,2^{-n}]$. Since $\tilde g$
has Lipschitz constant $5$, it follows that
$\BRA{x\in Q\,:\, \tilde g(x)>0}$ always contains an interval
of length larger than $\frac 1 5 2^{-n}$. As a consequence, since $X$ has a uniform distribution on $\Omega$, the failure probability associated to $\tilde g$ is such that 
$$
\dP(\tilde g(X)>0)\geq  c 2^{-n},
$$
with $c=\frac 1 5$.\medskip

This proves the optimality of the approximation error rate of our algorithm.

\section{Proof of Theorem  \ref{th:esti-p}}
\label{qlsjcnlnc}

Consistency and unbiasedness are clear by Remark \ref{lazclncskal}. The asymptotic normality is a consequence of the 
delta method. Remember that for $N\geq 1$ and $u,v\in\calT$, we denote
\[
C_N^v:=\sum_{i=1}^N \un_\BRA{X^{\bar v}_i\in Q(v)},
\quad 
q_N(v):=\frac{C^v_N}{N}
\quad\text{and}\quad 
p_N(u):=\prod_{v\in a(u)}q_N(v).
\]
First of all, let us recall the (classical) multidimensional CLT. 

\begin{lem}[Multidimensional CLT]
For all $w\in \Lambda$,  
\[
q_N(w)=\frac{C^w_N}{N}\xrightarrow[N\to\infty]{a.s.} q(w).
\]
Let us denote by $\mathbf{C}_N$ the random vector~${(C^w_N)}_{w\in \Lambda}$ and 
$\bq$ the vector~${(q(w))}_{w\in \Lambda}$. We have 
 \[
\sqrt{N}\SBRA{\frac{\mathbf{C}_N}{N}-\bq}
\xrightarrow[N\to\infty]{\mathcal{D}}\calN(0,\Gamma),
\]
where the covariance matrix $\Gamma$ is given by 
\[
\Gamma(v,w)=
\begin{cases}
 q(v)(1-q(v)) &\text{if }v=w,\\
 -q(v)q(w)&\text{if }v\neq w\text{ and }\bar v= \bar w,\\ 
 0&\text{otherwise.}
\end{cases}
\]
\end{lem}
From the latter we immediately deduce that
\[
p_N(\calS)=\sum_{u\in \calS}p_N(u)=\sum_{u\in \calS}\prod_{v\leq u}q_N(v)\xrightarrow[N\to\infty]{a.s.}p(\calS).
\]
Next, we may rewrite $p(\calS)$ as a function of $\bq={(q(v))}_{v\in \Lambda}$ as follows:
\[
p(\calS)=F(\bq):=\sum_{u\in \calS} \prod_{v\leq u}q(v). 
\]
The partial derivative of $F$ with respect to $q(v)$, denoted $\partial_vF$, is given by 
\[
\partial_v F(\bq)=\sum_{\substack{u\in\calS\\ v\leq u}} 
\prod_{\substack{w\in a(u)\\w\neq v}} q(w)
=\sum_{\substack{u\in\calS\\ v\leq u}}\frac{p(u)}{q(v)}. 
\]
If $\nabla F=(\partial_v F)_{v\in \Lambda}$ is seen a row vector, the delta method ensures that 
\[
\sqrt{N} (F(\bq_N)-F(\bq)) \xrightarrow[N\to\infty]{}\calN(0,\sigma^2), 
\]
where
\begin{align*}
\sigma^2&=(\nabla F) \Gamma (\nabla F)^T
=\sum_{v,w\in\Lambda} (\partial_vF)\Gamma(v,w)(\partial_wF)\\
&=\sum_{v\in\Lambda}\Gamma(v,v)(\partial_vF)^2
+\sum_{\substack{v\neq w\in\Lambda \\\bar v=\bar w}} \Gamma(v,w)(\partial_vF)(\partial_wF)\\
&=\sum_{v\in\Lambda}q(v)(1-q(v))\SBRA{\sum_{\substack{u\in \calS\\v\leq u}}\frac{p(u)}{q(v)}}^2
-\sum_{\substack{v\neq w\in \Lambda\\ \bar v=\bar w}}q(v)q(w)
\SBRA{\sum_{\substack{u\in \calS\\v\leq u}}\frac{p(u)}{q(v)}}
\SBRA{\sum_{\substack{u'\in \calS\\w\leq u'}}\frac{p(u')}{q(w)}}\\
&=\sum_{v\in\Lambda}\frac{1-q(v)}{q(v)}\SBRA{\sum_{\substack{u\in \calS\\v\leq u}}p(u)}^2
-\sum_{\substack{v\neq w\in \Lambda\\ \bar v=\bar w}}
\SBRA{\sum_{\substack{u\in \calS\\v\leq u}}p(u)}
\SBRA{\sum_{\substack{u'\in \calS\\w\leq u'}}p(u')}.
\end{align*}
Let us define 
\[
A:=\sum_{v\in\Lambda}\frac{1-q(v)}{q(v)}\SBRA{\sum_{\substack{u\in \calS\\v\leq u}}p(u)}^2
\quad\text{and}\quad 
B:=\sum_{\substack{v\neq w\in \Lambda\\ \bar v=\bar w}}
\SBRA{\sum_{\substack{u\in \calS\\v\leq u}}p(u)}
\SBRA{\sum_{\substack{u'\in \calS\\w\leq u'}}p(u')}.
\]
We have, since $(v\leq u) \Leftrightarrow v\in a(u)$, 
\begin{align*}
A&=\sum_{v\in\Lambda}\sum_{\substack{u\in \calS\\v\leq u}}p(u)^2\frac{1-q(v)}{q(v)}
+\sum_{v\in\Lambda}\sum_{\substack{u,u'\in \calS\\v\leq u\\ v\leq u'}}p(u)p(u')\frac{1-q(v)}{q(v)}\\
&=\sum_{u\in\calS}\sum_{v\in a(u)}p(u)^2\frac{1-q(v)}{q(v)}
+\sum_{\substack{u\neq u'\in \calS}}\sum_{v\in a(u)\cap a(u')}p(u)p(u')\frac{1-q(v)}{q(v)}.
\end{align*}
Similarly, we get 
\[
B=\sum_{\substack{v\neq w\in \Lambda\\ \bar v=\bar w}}
\sum_{\substack{u,u'\in \calS\\v\leq u\\w\leq u'}}p(u)p(u')
=\sum_{\substack{u\neq u'\in \calS}}p(u)p(u').
\]
Since $a(u)\cap a(u')=a(u\wedge u')$, this finally yields the claimed expression for $\sigma^2$.

\bibliographystyle{plain}
\bibliography{biblio-bcgm}
\end{document}

%% file: macros.tex
\newtheorem{defi}{Definition}[section]
\newtheorem{lem}{Lemma}[section]

\newcommand{\un}{{\mathbf{1}}}

\definecolor{darkpastelgreen}{rgb}{0.01, 0.75, 0.24}

\renewcommand{\leq}{\leqslant}

\renewcommand{\geq}{\geqslant}

\newcommand{\bq}{\mathbf{q}}

\newcommand{\N}{\mathbb{N}}

\newcommand{\R}{\mathbb{R}}

\renewcommand{\P}{\mathbb{P}}

\newcommand{\dN}{\mathbb{N}}
\newcommand{\dP}{\mathbb{P}}
\newcommand{\dR}{\mathbb{R}}

\newcommand{\calC}{\mathcal{C}}
\newcommand{\calD}{\mathcal{D}}

\newcommand{\calI}{\mathcal{I}}

\newcommand{\calL}{\mathcal{L}}

\newcommand{\calN}{\mathcal{N}}
\newcommand{\calO}{\mathcal{O}}

\newcommand{\calS}{\mathcal{S}}
\newcommand{\calT}{\mathcal{T}}
\newcommand{\calU}{\mathcal{U}}

\newcommand{\BRA}[1]{{{\left\{#1\right\}}}} 
\newcommand{\NRM}[1]{{{\left\| #1\right\|}}} 
\newcommand{\PAR}[1]{{{\left(#1\right)}}} 
\newcommand{\SBRA}[1]{{{\left[#1\right]}}} 
\newcommand{\ABS}[1]{{{\left| #1 \right|}}} 

\theoremstyle{plain}
\newtheorem{The}{Theorem}[section]
\newtheorem{Lem}[The]{Lemma}
\newtheorem{Pro}[The]{Proposition}

\newtheorem{Ass}{Assumption}

\newtheorem{Def}[The]{Definition}

\numberwithin{equation}{section}

\theoremstyle{definition}
\newtheorem{Rem}[The]{Remark}